\documentclass[a4paper,11pt]{amsart}
\makeatletter
\def\author@andify{%
  \nxandlist {\unskip ,\penalty-1 \space\ignorespaces}%
    {\unskip {} \@@and~}%
    {\unskip \penalty-2 \space \@@and~}%
}
\makeatother

\usepackage{graphicx}
\usepackage[small,sc]{caption}

\usepackage[T1]{fontenc}
\usepackage{lmodern, amsfonts,amsmath,amstext,amsbsy,amssymb,
amsopn,amsthm,upref,eucal,mathptmx,mathtools,url}

\usepackage{mathrsfs}

\RequirePackage{xcolor} 
\definecolor{halfgray}
{gray}{0.55}
\definecolor{webgreen}
{rgb}{0,0.4,0}
\definecolor{webbrown}
{rgb}{.8,0.1,0.1}
\definecolor{red}
{rgb}{1,0,0}
\usepackage{microtype}
\usepackage{tikz}

\newcommand \R {{ \mathbb R}}
\def\C{{\mathbb C}}

\newcommand \Z {{ \mathbb Z}}
\newcommand \N {{ \mathbb N}}
\newcommand \T {{ \mathbb T}}

   \newcommand{\tg}{\tilde{g}}
   \newcommand{\tv}{\tilde{v}}
   \newcommand{\tR}{R}
   
   \newcommand{\ep}{\varepsilon}
   \newcommand{\Ws}{W^s}
   
   \newcommand{\Wu}{W^u}
      \newcommand{\nuav}{\nu_{\av}}
    \newcommand{\fsym}{f_{\text{sym}}}
    \newcommand{\Fsym}{F_{\text{sym}}}
    \newcommand{\hsym}{h_{\text{sym}}}
    \newcommand{\Om}{\Omega}
        \newcommand{\Es}{E^s}
    
    \newcommand{\Eu}{E^u}

\newcommand*{\diff}{\mathop{}\!\mathrm{d}}
\newcommand{\one}{{\rm 1\mskip-4mu l}}
\newcommand{\norm}[1]{\left\lVert#1\right\rVert} 
\newcommand{\ior}{\operatorname{int}}
\newcommand{\Ccla}{\mathscr{C}}

\newcommand{\dist}{%
\operatorname{dist}
}

\DeclareMathOperator{\Leb}{Leb}

\DeclareMathOperator{\Lip}{Lip}
\DeclareMathOperator{\Max}{Max}
\DeclareMathOperator{\TV}{TV}
\DeclareMathOperator{\LF}{LF}
\DeclareMathOperator{\cov}{cov}
\DeclareMathOperator{\av}{av}

\newtheorem{theorem}{Theorem}[section]
\newtheorem {lemma}[theorem]{Lemma}
\newtheorem {proposition}[theorem]{Proposition}
\newtheorem{corollary}[theorem]{Corollary}
\newtheorem{remark}[theorem]{Remark}

\newtheorem{definition}[theorem]{Definition}

\newtheorem{question}[theorem]{Question}

\author{Paolo Giulietti}
\thanks{P.G.~Centro di Ricerca Matematica 
Ennio de Giorgi, Scuola Normale Superiore, Piazza dei 
Ca\-valieri 7, 56126 Pisa, Italy.
E-mail: \texttt{paolo.giulietti@sns.it}}
\author{Andy Hammerlindl}
\thanks{A.H.~School of Mathematics, Monash 
University, Victoria 3800 Australia. \\
E-mail: \texttt{andy.hammerlindl@monash.edu} }
\author{Davide Ravotti}
\thanks{D.R.~School of Mathematics, Monash 
University, Victoria 3800 Australia. \\
E-mail: \texttt{davide.ravotti@monash.edu}}

\title{Quantitative global-local mixing for accessible skew products}
 
\begin{document}

\begin{abstract}
We study global-local mixing for a family of accessible skew products with an exponentially mixing base and non-compact fibers, preserving an infinite measure. 
For a dense set of almost periodic global observables, we prove rapid mixing; and for a dense set of global observables vanishing at infinity, we prove polynomial mixing. More generally, we relate the speed of mixing to the \lq\lq low frequency behaviour\rq\rq\ of the spectral measure associated to our global observables.
Our strategy relies on a careful choice of the spaces of observables and on the study of a family of twisted transfer operators. 
\end{abstract}

\maketitle 

\setcounter{tocdepth}{1}
\tableofcontents

\section{Introduction}

Dynamical systems within the category of skew products have a long history.  They receive attention for many different reasons: in earlier ergodic theory,  
they were studied as mild generalizations of suspensions which cannot be 
factored \cite[Chapter 10]{CFS},  providing examples of 
simple 
partially hyperbolic systems; in recent days, they are often used to model real 
life situations {where fast-slow dynamics can be observed, such as in the study of climate} \cite{GhLu, BeSi}. As their name suggests, they are products of a base dynamics and a fiber dynamics. 
Here, we are concerned with the statistical properties of these systems.

One of the simplest examples of a partially hyperbolic skew product is given by 
\emph{circle extensions of Anosov diffeomorphisms} on the 2-dimensional torus 
$\T^2$, defined as follows. Given an Anosov diffeomorphism 
$A\colon \T^2 \to \T^2$ and a smooth map $f \colon \T^2 \to \T$, a skew 
product $F \colon \T^3 \to \T^3$ over $A$ induced by $f$ is defined by $F(x,r)= 
(Ax, r+f(x))$. The torus $\T^3$ is equipped with a product measure $\mu_u \times 
\text{Leb}$, where $\mu_u$ is any Gibbs measure with H{\" o}lder potential $u$ 
and $\text{Leb}$ is the Lebesgue measure on the fibers\footnote{Recall that if one chooses the potential as the $\det(DA|_{E_s})$, one recovers the usual SRB measure.}.
In this case, Dolgopyat \cite{Dol} proved that generic functions $f$ induce 
 skew products $F$ with \emph{rapid decay of correlations}, or \emph{rapid 
mixing}, i.e.~decorrelation of $\mathscr{C}^{\infty}$-observables 
is faster than any given polynomial. 
The speed of mixing may, in fact, be exponential, but this is still an open problem (\cite[Problem 
2]{Dol}). Dolgopyat's result holds in general for compact 
group extensions. 
The interested reader can also check the introductions of 
\cite{PaPo,BuEs, GRS} for an overview of old and new results on skew 
products. 

In this paper, we are interested in skew products with \emph{non-compact} 
fibers. In particular we will consider $\R$-extensions of topologically mixing 
Anosov diffeomorphisms { via their symbolic counterparts}.
For an introduction to infinite ergodic theory, we refer the reader to Aaronson's book \cite{Aar}.
In our setting, Guivarc'h showed that any H{\" o}lder function $f$ with zero integral which is not cohomologous to a constant induces an ergodic skew product \cite{Gui} (see also \cite{CLP}).

Concerning stronger statistical properties, 
an historical perspective on the various possible definitions of \emph{mixing}
in the infinite ergodic setting can 
be found in \cite{Len1}. 
We will be interested in \emph{global-local mixing}, a notion introduced by Lenci in \cite{Len1}, namely we will study the correlations between global and local observables  (see Definition \ref{def:globallocal}).
Local observables are akin to compactly supported observables, while 
global observables are supported over most of the of phase space.
One possible concern about the notion of 
global-local mixing may be the seemingly arbitrary 
choice of the averaging involved in the definition of global observables. 
In our setting, the infinite volume average is  analogous to statistical infinite volume limits introduced and refined along 
the years  by Van Hove, Fisher \cite[Section 3.3]{Mcc}  and Ruelle 
\cite[Section 3.9]{Rue},  which built on the inspirational work of  Bogoliubov~\cite{Bog}. 
Global-local mixing has been studied in different situations, for example random 
walks \cite{Len1,Len2}, mechanical systems \cite{DoNa, DLN}  and one dimensional  parabolic systems \cite{BLG}.

The main ingredient needed to study our class of skew product is \emph{accessibility} (see, among others, \cite{BW, BPSW, HHU}).  The skew-product $F$ is accessible if, roughly speaking, it is possible to reach any point in the space by moving along segments of stable and unstable manifolds. 

From the measure-theoretic point of view, using Markov partitions, we can translate the problem to study a skew product over a subshift of finite type, keeping the same $\R$-fibers. The question whether accessibility is preserved passing to symbolic dynamics appears to be delicate, and will be discussed in Appendix~\ref{sec:AnosovToSym}.
Our main result, Theorem \ref{thm:main1}, provides quantitative 
estimates for the decay of correlations of global and local observables for an accessible $\R$-extension of a symbolic shift.
To the best of our knowledge, this is the first quantitative result in the context of global-local mixing.

Contrary to the case of compact group extensions, we cannot expect exponential mixing in general, since, taking Fourier transforms, we have to deal with arbitrary low frequencies. Indeed, we will show in Theorem \ref{thm:main1} that the speed of convergence of  correlations depends on the behavior near zero of the spectral measure associated to the global observable (namely, its inverse Fourier-Stieltjes transform):  
if the support of this measure intersects a neighbourhood of $0$ only at $0$, then mixing is rapid (Theorem \ref{thm:rapid_mixing}) as we expect from Dolgopyat's result; in other cases we obtain polynomial estimates 
(Theorem \ref{thm:poly_mixing_1}, Theorem \ref{thm:poly_mixing_2}), which
correspond to the expected behaviour (see Remark \ref{remark:optimal}). Note that since the infinite 
volume average equals the 
value of the associated spectral measure at zero, our choice of infinite volume average is natural.

The main novelties of the work rely on a careful 
choice of the functional spaces involved. The accessibility hypothesis,  when coupled with a standard central limit theorem for the underlying symbolic 
dynamics on the basis, allows for transfer operator bounds on suitable splitting of lower 
and higher frequency modes and the exploiting of cancellation effects due to accessibility.

\subsection{Krickeberg mixing}

Another independent notion of mixing for infinite measure preserving transformation is known as \emph{Krickeberg} (or \emph{local}) \emph{mixing}.
An infinite measure preserving system $(X, \mu, T)$ is said to be Krickeberg mixing if there exists a sequence of positive numbers $\rho_n \to \infty$ such that, for any pair of \lq\lq nice\rq\rq\ finite measure sets $A,B$ (precisely, bounded sets whose boundary has zero measure), the rescaled correlations converge, that is 
\begin{equation}\label{eq:Kri_def}
\lim_{n \to \infty} \rho_n \, \mu(T^{-n}A \cap B) = \mu(A) \mu(B).
\end{equation}
The first example of a system satisfying this property dates back to Hopf in 1937 \cite{Hop}, but it was overlooked for several years, until the 60s when Krickeberg proposed \eqref{eq:Kri_def} as a definition of infinite measure mixing \cite{Kri}.
Since then, it has received considerable attention and Krickeberg mixing has been proved in several situations, see, e.g., \cite{DoNa2, Gou2, MeTe1, MeTe2, Mel, Pen} to name a few.

In the language of this paper, Krickeberg mixing can be seen as a remarkably strong form of \lq\lq local-local mixing\rq\rq: the correlations of local observables converge to zero and the first order term is the same up to a constant for any two sufficiently smooth compactly supported functions.
On the other hand, global-local mixing is a softer property, which however gives information on the correlations for a much wider class of observables, 
in particular for those which are not supported on a compact subset but \lq\lq see\rq\rq\ the whole phase space.

Exploiting some variations on the methods developed in this paper, we are able to establish a strong form of Krickeberg mixing for the accessible skew-products we consider here, namely we can prove a full asymptotic expansion of the correlations of Schwartz observables, in the spirit of the main theorem of \cite{DoNaPe}. 
This result will appear in a separate paper.

\subsection{Outline of the paper}

The rest of the paper is organized as follows. 
In Section \ref{sec:main} we rigorously introduce our framework and state our main results. 
In Section \ref{sec:the_main_result}, we describe in detail the classes of global and local observables we consider and we state our core result, Theorem \ref{thm:main1}. We then deduce Theorems \ref{thm:rapid_mixing}, \ref{thm:poly_mixing_1} and \ref{thm:poly_mixing_2} from Theorem \ref{thm:main1}.

In Section \ref{sec:one_side}, we present a 
preliminary result in the non-invertible case of skew products over one-sided subshifts, Theorem \ref{thm:main2}. We also describe a \lq\lq collapsed accessibility\rq\rq\ property, which constitutes the main working assumption on the skew product in this setting. 

The main tool to prove Theorem \ref{thm:main2} is a family of twisted transfer operators. 
In Section \ref{sec:twisted_transfer_op}, we show how the collapsed accessibility property can be exploited to obtain some cancellations in the expression for the twisted transfer operators, as in the work of Dolgopyat \cite{Dol}. 
In Section \ref{sec:contraction}, we prove some estimates on the norm of the twisted transfer operators. For large twisting paramenters, the estimates are obtained exploiting the results in Section \ref{sec:twisted_transfer_op}; for small parameters, we apply some standard results in the theory of analytic perturbations of bounded linear operators.

Section \ref{sec:decay} contains some technical results that will be applied to prove the main theorems.
Section \ref{sec:pf_main2} is devoted to the proof of Theorem \ref{thm:main2}.
In order to deduce Theorem \ref{thm:main1} from Theorem \ref{thm:main2}, in Section \ref{sec:acc_collapsed_acc} we deduce the collapsed accessibility property for a one-sided skew-product from the accessibility property of the corresponding two-sided skew-product. In Section \ref{sec:proof_thm_1}, we prove Theorem \ref{thm:main1}.

In Appendix \ref{sec:AnosovToSym}, we discuss the problem whether the accessibility property for a skew product over an Anosov diffeomorphism is equivalent to the accessibility of the associated symbolic system.
Appendices \ref{sec:proofs_of_lemmas} and \ref{sec:appendix2} contain the proofs of several technical results.

\section{Setup and main results} \label{sec:main}

Let $\sigma \colon \Sigma \to \Sigma$ be a topologically mixing two-sided subshift of finite type, equipped with a Gibbs measure $\mu = \mu_u$ with respect to a H\"older potential $u$ (see, e.g., \cite[\S1]{BowenBook} or \cite[\S3]{PaPo}). 
{Explicitly, for $ x \in \Sigma $, $x= \{ x_i \}_{i \in \mathbb{Z}}$, we have $(\sigma x)_i = x_{i+1}$.}
For  $0<\theta < 1$, define the distance 
$$
d_\theta(x,y) = \theta^{\max\{j \in \N\ :\ x_i = y_i \text{\ for all\ } |i|<j\}}.
$$
Let us denote by $\mathscr{F}_{\theta}$ the space of Lipschitz continuous functions $w:\Sigma \to \C$ equipped with the norm
\begin{equation}\label{eq:def_Lip_seminorm}
\norm{w}_{\theta} = \norm{w}_{\infty} + |w|_{\theta}, \text{\ \ \ where\ \ \ } |w|_{\theta} = \sup_{x\neq y} \frac{|w(x)-w(y)|}{d_{\theta}(x,y)}.
\end{equation}
We consider the skew-product
\begin{equation}\label{eq:skew_product}
F\colon \Sigma \times \R \to \Sigma \times \R, \ \ \ \ \ F(x, r) = (\sigma 
x, r + f(x)),
\end{equation}
where $f\colon \Sigma \to \R$ is a Lipschitz continuous function with zero average, $\int_\Sigma f \diff \mu_u = 0$. 

{As mentioned in the introduction}, we will assume that $F$ is \emph{accessible}: {roughly speaking, this means that any two points can be connected by a path consisting of pieces of stable and unstable manifolds;} see Section \ref{sec:the_main_result} for the precise definition.

We are interested in the mixing properties of the map $F$ with respect to the \emph{infinite} measure $\nu =\mu \times \Leb$, where $\Leb$ is the Lebesgue measure on $\R$.

\begin{definition} \label{def:globallocal}   A \emph{local observable} is any function $\psi \in L^1(\nu)$. A \emph{global observable} is any function $\Phi \in L^{\infty}(\nu)$ such that the following limit exists 
\begin{equation}\label{eq:lim_defin}
\nuav(\Phi) := \lim_{R \to \infty} \frac{1}{2R} \int_{\Sigma \times [-R,R]} \Phi(x,r) \diff \nu(x,r).
\end{equation}
\end{definition}
If $\psi$ is a local observable, we will write $\nu(\psi) = \int_{\Sigma \times \R} \psi \diff \nu$. 
We will show in Lemma \ref{lemma_app_1a} below that if $\Phi$ is a global observable, then so is $\Phi \circ F$, and the \emph{average} $\nuav$ defined in \eqref{eq:lim_defin} is invariant under $F$.

For any pair of global and local observables $(\Phi,\psi)$, let us denote by $\cov(\Phi, \psi)$ the  covariance
$$
\cov(\Phi, \psi) :=  \nu( \Phi \cdot \overline{\psi} ) -  \nuav(\Phi) \nu( \psi).
$$

We are interested in showing ``\emph{global-local mixing}'', namely in proving that the correlations $\cov(\Phi \circ F^n, \psi)$ satisfy
\begin{equation} \label{eq:glob-local}
\lim_{n \to \infty} \cov(\Phi \circ F^n, \psi) = 0.
\end{equation}
and the rate of convergence to such limit (also known as the rate of decay of correlations).

The main result of this paper, Theorem \ref{thm:main1} below, establishes \emph{quantitative global-local mixing estimates}.
Since some preliminary work is needed, the statement of the main theorem is postponed to Section \ref{sec:the_main_result}.
We state here some corollaries which should give the reader a rather complete picture of the possible scenarios. 
Theorem \ref{thm:rapid_mixing} states that, for a dense class of \emph{almost periodic}\footnote{We recall that the space of \emph{almost periodic functions} is the closure of the space of trigonometric polynomials with respect to the uniform norm.} global observables, we have \emph{rapid mixing}, namely the decay of correlations is faster than any given polynomial, in analogy to Dolgopyat's result \cite{Dol} in the case of circle extensions. 
On the other hand, for a dense class of global observables which vanish at infinity, Theorems \ref{thm:poly_mixing_1} and \ref{thm:poly_mixing_2} state that the decay is polynomial. 
The bound in Theorem \ref{thm:poly_mixing_2} is generically \emph{optimal}, as we show in \S\ref{sec:example_optimal}.

Let us fix some notation. Let  $\Ccla^k(\R)$ be the space of $k$-times differentiable functions on $\R$.
We will denote by $\mathscr{P}$ the subspace of $\mathscr{C}^2(\R)$ consisting of $2\pi$-periodic functions; by $\mathscr{C}_0(\R)$ the space of continuous functions on $\R$ which vanish at infinity, and by $\mathscr{C}^\infty_c(\R)$ the subspace of infinitely differentiable functions with compact support.
The space $\mathscr{C}^\infty_c(\R)$ has the structure of a Fr\'{e}chet space, induced by the family of seminorms $\|\cdot\|_{\mathscr{C}^k}$, for $k \in \N$. We will say that a map $\psi \colon \Sigma \to \mathscr{C}^\infty_c(\R)$ is Lipschitz if it is Lipschitz with respect to $\|\cdot\|_{\mathscr{C}^k}$, for all $k \in \N$.

In the rest of the paper, we will implicitely identify maps $a$ from $\Sigma$ to some space of complex-valued measurable functions over $\R$ with complex-valued measurable functions on $\Sigma \times \R$ by setting $a(x,r) = [a(x)](r)$.

\begin{theorem}[Rapid global-local mixing]\label{thm:rapid_mixing}
Assume that $F$, defined as in \eqref{eq:skew_product}, is accessible. 
For any Lipschitz map $\psi \colon \Sigma \to \mathscr{C}^\infty_c(\R)$, for any Lipschitz map $\Phi \colon \Sigma \to \mathscr{P}$,
and for every $\ell \in \N$, there exists a constant $C=C(\ell, \psi, \Phi)\geq 0$ such that for all $n \in \N$,
$$
\left\lvert \cov(\Phi \circ F^n, \psi) \right\rvert \leq C n^{-\ell}.
$$
\end{theorem}
{Rapid mixing holds in many more situations besides the one in Theorem \ref{thm:rapid_mixing}. In a precise sense that will become clear later, the key property of the global observable that ensures rapid mixing is the absence of \lq\lq low frequencies components\rq\rq\ in almost every fiber, a property which is clearly satisfied by smooth functions which are periodic on the fibers.}

The situation is different when the global observable vanishes at infinity. Note that, in this case, the average $\nuav$ defined in \eqref{eq:lim_defin} is zero.

\begin{theorem}[Polynomial global-local mixing I]\label{thm:poly_mixing_1}
Assume that $F$, defined as in \eqref{eq:skew_product}, is accessible. 
There exists a space of bounded continuous functions $\mathscr{D} \subset  \mathscr{C}_0(\R)$, which is dense in $\mathscr{C}_0(\R)$ with respect to $\|\cdot\|_{\infty}$, and there exists $\alpha >0$ such that the following holds. For any Lipschitz map $\psi \colon \Sigma \to \mathscr{C}^\infty_c(\R)$, and for any map $\Phi \colon \Sigma \to \mathscr{D}$ which satisfies some explicit Lipschitz condition, there exists a constant $C=C(\psi, \Phi)\geq 0$ such that for all $n \in \N$,
$$
\left\lvert \cov(\Phi \circ F^n, \psi)  \right\rvert \leq C n^{-\alpha}.
$$
\end{theorem}

The Lipschitz conditions for the global observable in Theorem \ref{thm:poly_mixing_1} will be stated explicitly in Section \ref{sec:the_main_result} below, as well as a bound on $\alpha$ { (see the second paragraph of the proofs at \S\ref{section:proofs_of_2.3_2.4})}.
If we further assume that the global observable takes values in $W^1(\R)$, where $W^1(\R)$ is the Sobolev space of $L^2$ functions with weak derivative in $L^2$, then the statement reads as follows.

\begin{theorem}[Polynomial global-local mixing II]\label{thm:poly_mixing_2}
For any Lipschitz map $\psi \colon \Sigma \to \mathscr{C}^\infty_c(\R)$, for any Lipschitz map $\Phi \colon \Sigma \to W^1(\R)$, and for any $\ep >0$, there exists a constant $C=C(\psi, \Phi, \ep)\geq 0$ such that for all $n \in \N$,
$$
\left\lvert \cov(\Phi \circ F^n, \psi)  \right\rvert \leq C n^{-\frac{1}{4}+\ep}.
$$
If moreover $\Phi \colon \Sigma \to W^1(\R) \cap L^p(\R)$ for some $1 \leq p \leq 2$, then 
$$
\left\lvert \cov(\Phi \circ F^n, \psi)  \right\rvert \leq C n^{-\frac{1}{2p}+\ep}.
$$
\end{theorem}

\begin{remark}\label{remark:optimal}
The bound in Theorem \ref{thm:poly_mixing_2} is optimal: we will provide an example in \S\ref{sec:example_optimal} of a pair of global and local observables $\Phi, \psi$ for which the correlations are bounded below by $|  \cov(\Phi \circ F^n, \psi)  | \geq B n^{-\frac{1}{2}}$ for some constant $B>0$, and, on the other hand, Theorem \ref{thm:poly_mixing_2} implies that for any $\ep >0$ there exists a constant $C>0$ such that $| \cov(\Phi \circ F^n, \psi)  | \leq C n^{-\frac{1}{2}+\ep}$.
\end{remark}

{ We stated our results on the symbolic systems, as it makes the statements easier to read. Thanks to the semiconjugacy between the diffeomorphism and the symbolic dynamics, analogous classes of observables give analogous results on the original dynamics, as their regularity is preserved (see Lemma \ref{lemma_app_1a} for the details).}

\section{The main result}\label{sec:the_main_result}

In this section, we first recall the definition of accessibility for $F$ as in \eqref{eq:skew_product}; then, we describe the classes of global and local observables we consider, and we state our main result. We deduce Theorems \ref{thm:rapid_mixing}, \ref{thm:poly_mixing_1} and \ref{thm:poly_mixing_2} from Theorem \ref{thm:main1}. { To help the reader in following the flow of the proofs, some of the lemmas stated in this section are proved in  Appendix \ref{sec:proofs_of_lemmas}.}

\subsection{Accessibility}

For each point $x\in \Sigma$, we define the \emph{stable} and \emph{unstable set} at $x$ by, respectively,
\begin{equation*}
\begin{split}
& W^s(x) = \{ y \in \Sigma : \text{ there exists }n \in \Z \text{ such that } y_i=x_i \text{ for all } i \geq n\}, \\
& W^u(x) = \{  y \in \Sigma : \text{ there exists }n \in \Z \text{ such that } y_i=x_i \text{ for all } i \leq n \}.
\end{split}
\end{equation*}
By definition, for any $y \in W^s(x)$, $d_\theta(\sigma^nx, \sigma^ny) \to 0 $ exponentially fast and, similarly, for $y \in W^u(x)$, $d_\theta(\sigma^{-n}x, \sigma^{-n}y) \to 0 $; moreover, note that $d_\theta$ attains a discrete set of values $\{ \theta^{i} \}_{i \in \N}$.

The skew product \eqref{eq:skew_product} is \emph{partially hyperbolic} in the following sense. Let us denote by $f_n(x) = \sum_{i=0}^{n-1} f \circ \sigma^i(x)$ the $n$-th Birkhoff sum at $x$. Let us define for any $(x,r) \in  \Sigma \times \R$
\begin{equation*}
\begin{split}
& W^s(x,r) = \{ (y,s) \in \Sigma \times \R: y \in W^s(x) \text{ and } s-r = \lim_{n \to \infty} f_n(x) - f_n(y)\}, \\
& W^u(x,r) = \{ (y,s) \in \Sigma \times \R: y \in W^u(x) \text{ and } s-r = \lim_{n \to \infty} f_n(\sigma^{-n}y) - f_n(\sigma^{-n}x)\}.
\end{split}
\end{equation*}
We equip $ \Sigma \times \mathbb{R}$ with the product distance given by 
\[
\dist\big((x,r), (y,s) \big) = d_\theta(x,y) + |s-t|;
\]
it is easy to see that  
\begin{equation*}
\begin{split}
& \lim_{n \to \infty} \dist(F^n(x,r), F^n(y,s)) = 0 \text{ exponentially fast, 
if }(y,s) \in W^s(x,r), \\
& \lim_{n \to -\infty} \dist(F^n(x,r), F^n(y,s)) = 0 \text{ exponentially fast, 
if }(y,s) \in W^u(x,r).
\end{split}
\end{equation*}
The sets $W^s(x,r)$ and $W^u(x,r)$ are called the \emph{(strong) stable} and \emph{(strong) unstable manifold} at $(x,r) \in \Sigma \times \R$. Vertical lines $\{x\} \times \R$ constitute the \emph{center manifolds}, namely they form 
an invariant fibration and the action of $F$ on each line is 
isometric.

We now define the accessibility property. A $su$-path from $(x,r)$ to $(y,s)$ is a finite sequence $(x^i,r_i)  \in \Sigma \times \R$, for $0 \leq i \leq m$ for some $m \in \N$, such that  $(x^0,r_0) = (x,r)$, $(x^m,r_m)= (y,s)$, and $(x^i,r_i) \in W^s(x^{i-1},r_{i-1})$ or $(x^i,r_i) \in W^u(x^{i-1},r_{i-1})$  for all $1 \leq i \leq m$. 
We say that $F$ is \emph{accessible} if for any two points $(x,r), (y,s) \in \Sigma \times \R$ there is a $su$-path from $(x,r)$ to $(y,s)$.

A consequence of the accessibility property is the following fact, which will be proved in Appendix \ref{sec:proofs_of_lemmas}.
\begin{lemma}\label{lemma:not_cohom_zero}
If $F$ is accessible, then $f$ is not cohomologous to zero.
\end{lemma} 

This lemma is used in Section \ref{sec:contraction} when analyzing the analytic
perturbation of the transfer operator. However, we still need to directly use
accessibility of $F$ in other parts of the proof.

\subsection{The classes of global and local observables}

We now describe the classes of global and local observables we consider, which we will call \emph{good} global and local observables.
Let us start by observing that the average defined in \eqref{eq:lim_defin} is invariant under $F$.
\begin{lemma}\label{lemma_app_1a}
If $\Phi$ is a global observable according to Definition \ref{def:globallocal}, then $\Phi \circ F$ is a global observable and $\nuav (\Phi \circ F)=\nuav(\Phi)$.
\end{lemma}
{The proof of Lemma \ref{lemma_app_1a}, which can be found in Appendix \ref{sec:proofs_of_lemmas}, is a consequence of the invariance of the Gibbs measure with respect to the dynamics.}

We will denote by $\mathscr{S}$ the Fr\'{e}chet space of Schwartz functions on $\R$, with the family of seminorms 
$$
\|g\|_{a,\ell} := \sup_{r \in \R} |r|^a \left\lvert \frac{\diff^\ell}{(\diff r)^\ell} g(r)\right\rvert.
$$ 
We will say that a function $\psi \colon \Sigma \to \mathscr{S}$ is H{\" o}lder if it is H{\" o}lder with respect to $\| \cdot \|_{a,\ell}$ for all $a,\ell \in \N$. Starting from definition \eqref{def:globallocal}, we  restrict ourselves from now on to smaller classes of observables.
\begin{definition}[Good local observables]\label{def:L_observables}
We denote by $\mathscr{L}\subset L^1(\nu)$ the space of H{\" o}lder functions $\psi 
\colon \Sigma \to \mathscr{S}$. 
\end{definition}

Let $\eta$ be a \emph{complex measure} over $\R$. We will denote by $|\eta|$ the \emph{variation} of $\eta$ and by $\| \eta\|_{\TV} = |\eta|(\R)$ its \emph{total variation}.
We recall that the \emph{Fourier-Stieltjes transform} $\widehat{\eta}(r)$ of a complex measure $\eta$ of finite total variation is the $L^\infty$ function defined by 
$$
\widehat{\eta}(r):= \int_\R e^{-ir\xi} \diff \eta(\xi).
$$
{{Let $\mathscr{A}$ be the set of such $\widehat{\eta}$'s:} 
the space $\mathscr{A}$ of all Fourier-Stieltjes transforms is an algebra of functions, called the \emph{Fourier-Stieltjes algebra}. We equip $\mathscr{A}$ with the total variation norm, namely, for $\widehat{\eta}_1, \widehat{\eta}_2 \in \mathscr{A}$, we set
$$
\|\widehat{\eta}_1 - \widehat{\eta}_2\| := \|\eta_1 - \eta_2\|_{\TV}.
$$
\begin{definition}[Good global observables]\label{def:G_observables}
We denote by $\mathscr{G}\subset L^\infty(\nu)$ the space of H{\" o}lder functions $\Phi \colon \Sigma \to \mathscr{A}$ which satisfy the following \emph{tightness condition:}
\begin{equation}\label{eq:TC} \tag{TC}
\begin{split}
&\text{there exist } a,A >0 \text{ such that for all $ r \geq 1 $ and $x \in \Sigma$ we have } \\
&|\eta_x|( \R \setminus [-r,r]) \leq Ar^{-a}, 
\end{split}
\end{equation}
where $\widehat{\eta_x} = \Phi(x)$.
\end{definition}
{The tightness condition \eqref{eq:TC} above ensures that one can control the large frequency behaviour of $\Phi(x,\cdot)$ \emph{uniformly} in the point $x \in \Sigma$. In particular,
it will be exploited in the proofs of Lemma \ref{lemma:phi_h} and Proposition \ref{lemma:one_sided} in Appendix \ref{sec:appendix2} to have some compactness property. 

\begin{remark}
Any H{\" o}lder function from $\Sigma$ into a metric space can be made Lipschitz by choosing a larger $\theta$ and hence changing the metric $d_\theta$ on $\Sigma$. Since we are not imposing any condition on $\theta$, here and henceforth we will assume that good local observables are Lipschitz maps $\psi \colon \Sigma \to \mathscr{S}$ and good global observables are Lipschitz maps $\Phi \colon \Sigma \to \mathscr{A}$ satisfying \eqref{eq:TC}.
\end{remark}

The following lemma shows that elements of $\mathscr{G}$ are indeed global observables; more precisely, the average $\nuav$ of $\Phi$ is the average of the values of the associated measures $\eta_x$ at 0.
{See Appendix \ref{sec:proofs_of_lemmas} for the short complex integration computation which leads to the result.}
\begin{lemma}\label{lemma:nu_u}
If $\Phi \in \mathscr{G}$, then $\Phi$ is a global observable according to Definition \ref{def:globallocal} and 
$$
\nuav(\Phi) = \int_{\Sigma} \eta_x(\{0\}) \diff \mu(x),
$$
where, as before, $\widehat{\eta_x} = \Phi(x)$.
\end{lemma}

We conclude this section by providing a useful criterion to determine whether a given function is the Fourier-Stieltjes transform of a finite complex measure which satisfies \eqref{eq:TC}.
A \emph{positive definite function} is any function $g \colon 
\R \to \C$ such that 
$$
\sum_{i,j=1}^ng(x_i-x_j)z_i\overline{z_j} \geq 0,
$$
for all $n \geq 1$, $x_i,x_j \in \R$ and $z_i, z_j \in \C$.
By Bochner's theorem, a function $g$ is continuous  and positive definite if and only if it is the Fourier-Stieltjes transform $\widehat{\eta}$ of a finite positive measure $\eta$ on $\R$ (see, e.g., \cite[Theorem IX.9]{RS75}). 
For example, it is easy to check that $g(x) = e^{ix}$ or $g(x)=\cos(x)$ are positive definite functions.
A less trivial example is the function $g(x)=\frac{1}{|x|+1}$; the fact that $g$ is positive definite follows from P\'{o}lya's Criterion: any positive, continuous, even function which, for positive $x$, is non-increasing, convex and tends to 0 for $x\to \infty$  is the Fourier-Stieltjes transform of an $L^1$ function, thus positive definite.
We refer the reader to \cite{LST} and \cite[Chapter 6]{TrBe} for more results concerning the Fourier-Stieltjes transform of measures.

\begin{lemma}
Any linear combination of Lipschitz positive definite functions is the Fourier-Stieltjes transform of a complex measure of finite total variation which satisfies \eqref{eq:TC}.
\end{lemma}
Since any complex measure of finite total variation is a linear combination of positive finite measures, the proof of the lemma above follows immediately from the following tail estimate, whose proof can be found in Appendix \ref{sec:proofs_of_lemmas}.
\begin{lemma}\label{lemma:appendix3}
Let $\eta$ be a finite positive measure on $\R$, let $\Phi(r)$ be its Fourier-Stieltjes transform. Then, if $\Phi(r)$ is Lipschitz of constant $L$, for all $r >0$ we have
$$
\eta( \R \setminus [-r,r]) \leq \frac{2L}{r}.
$$
\end{lemma}

\subsection{Statement of the main result}

For any good global observable $\Phi \in \mathscr{G}$ and for any $r >0$, let us define the ``low frequency variation'' as
\begin{equation}\label{defin:LF}
\LF(\Phi, r) := \int_{\Sigma} |\eta_x| \big( (-r, r) \setminus \{0\} \big) \diff \mu(x).
\end{equation}
Notice that $\LF(\Phi, \cdot)$ is monotone and $\LF(\Phi, r) \to 0$ for $r \to 0$.
We are now ready to state our main result.
\begin{theorem}[Quantitative global-local mixing]\label{thm:main1}
Assume that $F$, defined as in \eqref{eq:skew_product}, is accessible. Then, for every $\psi \in \mathscr{L}$, for every $\Phi \in \mathscr{G}$, for any $k \in \N$, and for every $\ep >0$, there exists a constant $C=C(\Phi, \psi, k, \ep)>0$ such that for every $n \in \N$, 
$$
|\cov(\Phi \circ F^n, \psi)| \leq C \left( \LF \Big(\Phi, n^{-\frac{1}{2}+ \ep} \Big) + n^{-k} \right).
$$
\end{theorem}
The bound in Theorem \ref{thm:main1} is the sum of two terms, namely a \emph{superpolynomial term} and the contribution given by the measures $|\eta_x|$ close to $0$. In particular, if the support of the measures $\eta_x$ does not intersect some neighbourhood of $0$, then the decay of correlations is superpolynomial. On the other hand, for example under the assumptions of Theorem \ref{thm:poly_mixing_2}, the measures $|\eta_x|$ are absolutely continuous and the decay is polynomial.

In the rest of the section we prove Theorems \ref{thm:rapid_mixing}, \ref{thm:poly_mixing_1} and \ref{thm:poly_mixing_2} from the result above.

\subsection{Proof of Theorem \ref{thm:rapid_mixing}}
We deduce Theorem \ref{thm:rapid_mixing} from Theorem \ref{thm:main1}.

By the theory of Fourier series, any $p \in \mathscr{P} \subset \mathscr{C}^2(\R)$, by periodicity, is the Fourier-Stieltjes transform of a discrete measure $\eta$ of the form $\eta=\sum_{n \in \Z} a_n \delta_n$, where $a_n \in \C$ and $\delta_n$ is the Dirac measure at $n$. 
We claim that any Lipschitz map $\Phi \colon \Sigma \to \mathscr{P}$ is contained in $\mathscr{G}$. Theorem \ref{thm:main1} then immediately implies the result, since $|\eta_x| ( (-1,1) \setminus \{0\} )=0$ (where, as usual, we write $\Phi(x) = \widehat{\eta_x}$). 

We first check the Lipschitz condition. For $x \in \Sigma$, let us write $\eta_x = \sum_{n \in \Z} a_n(x) \delta_n$. 
Since $\Phi(x) \in \mathscr{C}^2(\R)$, it follows that $\lim_{n \to \infty}|n^2 a_n(x)| = 0$. 
In particular, the sequence $|a_n(x)| \cdot (1+i|n|)$ is square-summable (notice that $ina_n(x)$ are the Fourier coefficients of the derivative $\Phi(x)'$).
Thus, for any $x,y \in \Sigma$, by Cauchy-Schwartz, we have
\begin{equation*}
\begin{split}
\|\Phi(x) - \Phi(y)\| &= \|\eta_x - \eta_y\|_{\TV} = \sum_{n \in \Z} |a_n(x)-a_n(y)| = \sum_{n \in \Z} |a_n(x)-a_n(y)| \cdot \frac{1+i|n|}{1+i|n|} \\
&\leq \left( \sum_{n \in \Z} \frac{1}{1+n^2} \right)^{\frac{1}{2}} \cdot \left( \sum_{n \in \Z} |a_n(x)-a_n(y)|^2 + |a_n(x)-a_n(y)|^2 \cdot n^2 \right)^{\frac{1}{2}}.
\end{split}
\end{equation*}
Hence, by Plancharel formula, there exists a constant $C >0$ such that 
$$
\|\Phi(x) - \Phi(y)\| \leq C \left(\|\Phi(x) - \Phi(y)\|_\infty + \|\Phi(x)' - \Phi(y)'\|_\infty\right)  \leq  C \|\Phi(x) - \Phi(y)\|_{\mathscr{C}^2}.
$$
This shows that $\Phi \colon \Sigma \to \mathscr{G}$ is Lipschitz.

We now show that $\Phi$ satisfies the tightness condition \eqref{eq:TC}.
Since $\Phi(x) \in \mathscr{C}^2(\R)$, we can bound $|a_n(x)| \leq \|\Phi(x)''\|_\infty n^{-2}$. Thus, for any $r \geq 2$ we have
$$
|\eta_x| (\R \setminus [-r,r]) = \sum_{|n| > r} |a_n(x)|\leq \|\Phi(x)''\|_\infty \sum_{|n| > r} n^{-2}\leq \|\Phi(x)\|_{\mathscr{C}^2} \, r^{-1},
$$
which concludes the proof.

\subsection{Proof of Theorems \ref{thm:poly_mixing_1} and \ref{thm:poly_mixing_2}}\label{section:proofs_of_2.3_2.4}

Let us first prove Theorem \ref{thm:poly_mixing_1}. 
To this end, fix any $p>1$ and consider as $\mathscr{D}$ the space of Fourier transforms of functions $f \in L^1 \cap L^p$ with power decay, namely, for which there exist constants $A,a >0$ such that $f(\xi) \leq A |\xi|^{-a}$ for all $|\xi| \geq 1$. 
Then, since $\mathscr{S} \subset \mathscr{D} \subset \mathscr{G}$, it is clear that $\mathscr{D}$ is dense in $\mathscr{C}_0(\R)$.

Consider $\Phi \colon \Sigma \to \mathscr{D}$, and write $\Phi(x) = \widehat{\eta_x}$ where $\diff \eta_x = f_x(\xi) \diff \xi$, with $f_x \in L^1 \cap L^p$. The Lipschitz condition we need to impose on $\Phi$ to be a good global observable is $\|\eta_x - \eta_y\|_{\TV} = \|f_x - f_y\|_{L^1} \leq L(\Phi) \, d_{\theta}(x,y)$ for some constant $L(\Phi) >0$.
In order to conclude, we show that $|\eta_x|\big( -n^{-\frac{1}{2}+ \ep}, n^{-\frac{1}{2}+ \ep}\big)$ decays as a power of $n$ for all $x \in \Sigma$. 
Let $0<{\tilde \alpha} = \frac{1}{2} - \ep< \frac{1}{2}$. Then, using H\"older inequality, with $\frac{1}{p}+\frac{1}{q} = 1$, 
\begin{equation*}
\begin{split}
|\eta_x|\big( -n^{- {\tilde \alpha} }, n^{- {\tilde \alpha} } \big) &\leq \int_{-n^{- {\tilde \alpha} }}^{n^{- {\tilde \alpha} }} |f_x|(\xi)\diff \xi \leq \|f_x\|_p \left\|\one_{(-n^{- {\tilde \alpha} }, n^{- {\tilde \alpha} })} \right\|_q \\
& \leq \|f_x\|_p (2n^{- {\tilde \alpha} })^{1/q} \leq \|f_x\|_p (2n^{- {\tilde \alpha} })^{1-1/p}.
\end{split}
\end{equation*}
Therefore, Theorem \ref{thm:poly_mixing_1} holds for any $\alpha$ of the form ${\tilde \alpha} (1-\frac{1}{p}) = (\frac{1}{2} - \ep)(1-\frac{1}{p}) $, with $\ep>0$.

Let us now prove Theorem \ref{thm:poly_mixing_2}. 
It follows from \cite[Theorem 4.2]{Bob} that any function $f \in W^1$ is the Fourier transform of a function $g \in L^1 \cap L^2$, which satisfies $\|g\|_2 = \|f\|_2$ and $\|g\|_1 \leq \|f\|_{W^1}$.
This implies that any Lipschitz map $\Phi \colon \Sigma \to W^1$ is Lipschitz also with respect to the total variation norm. 

Let us check that $\Phi$ satisfies the tightness condition \eqref{eq:TC}. 
Denote $\diff \eta_x(\xi) = f_x \diff \xi$, with $f_x \in L^1 \cap L^2$. 
Again, it follows from \cite[Theorem 4.2]{Bob} that $\xi f_x(\xi) \in L^2$ and $ \| \xi f_x(\xi)\|_2 \leq \|\Phi(x)'\|_2$. For any $r \geq 2$, by Cauchy-Schwartz and by Plancharel formula, we have
\begin{equation*}
\begin{split}
|\eta_x| (\R \setminus [-r,r]) &= \int_{\R \setminus [-r,r]} |f_x|(\xi) \diff \xi = \int_{\R \setminus [-r,r]} |f_x|(\xi)  \frac{1+i|\xi|}{1+i|\xi|} \diff \xi \\
& \leq \left\| \one_{\R \setminus [-r,r]} \cdot (1+i|\xi|)^{-1} \right\|_2 \cdot \|(1+i|\xi|) \cdot |f_x|(\xi)\|_2 \\
& \leq (\pi - 2 \tan^{-1}(r)) \|\Phi(x)\|_{W^1} \leq \left(\max_{x \in\Sigma} \|\Phi(x)\|_{W^1} \right) r^{-1}.
\end{split}
\end{equation*}
The estimate $|\eta_x|\big( -n^{-\frac{1}{2}+ \ep}, n^{-\frac{1}{2}+ \ep} \big) = O \big( n^{-\frac{1}{4}+ \ep}\big)$ follows from Cauchy-Schwarz inequality exactly as above.  
If in addition $\Phi$ has range in $W^1 \cap L^p$, then the functions $f_x$ belong to $L^q$, where $\frac{1}{p} + \frac{1}{q} = 1$, and one can conclude using H\"older inequality again.
This finishes the proof.

\subsection{Example}\label{sec:example_optimal}

We discuss a simple example, which shows that the bound in Theorem \ref{thm:poly_mixing_2} cannot, in general, be improved. 
As good local observable, let us consider any non-negative $\psi(x,r) = \psi(r) \in \mathscr{C}^\infty_c(\R)$ which equals $1$ in the interval $\left[ -\frac{1}{2},\frac{1}{2} \right]$ and, as good global observable, let  $\Phi(x,r)= \Phi(r) = \frac{1}{1+|r|}$. 
Then, $\Phi \in W^1(\R) \cap L^p(\R)$ for any $p >1$, so that Theorem \ref{thm:poly_mixing_2} implies that for any $\ep >0$ there exists a constant $C\geq 0$ such that 
$$
\left\lvert \cov(\Phi \circ F^n, \psi) \right\rvert = \int_{\Sigma \times \R} (\Phi \circ F^n)\cdot \psi \diff \nu \leq C n^{-\frac{1}{2}+\ep}.
$$
Let us show that there is a lower bound of order exactly $O(n^{-\frac{1}{2}})$.

Lemma \ref{lemma:not_cohom_zero} implies that $f$ is not cohomologous to zero.
Moreover, by the Central Limit Theorem, there exists a constant $C'>0$ such that for any $n \in \N$ sufficiently large, on a subset $Y_n \subset \Sigma$ of measure at least $1/2$, the Birkhoff sums $f_n (x) = f(x) + \cdots + f(\sigma^{n-1}x)$ are bounded by $|f_n(x)| \leq C' \sqrt{n}$. In particular, for any $x \in Y_n$ and $ r \in \left[ -\frac{1}{2},\frac{1}{2} \right]$, we have
$$
\Phi \circ F^n(x,r) = \Phi(r+ f_n(x)) = \frac{1}{1+r+|f_n(x)|} \geq \frac{1}{2 C' \sqrt{n}}.
$$ 
Thus, for any $n \in \N$ sufficiently large, it follows that
\begin{equation*}
\begin{split}
\int_{\Sigma \times \R} (\Phi \circ F^n)\cdot \psi \diff \nu &\geq \int_{Y_n \times \left[ -\frac{1}{2},\frac{1}{2} \right]} (\Phi \circ F^n)\cdot \psi \ \diff \nu \geq \nu\left(Y_n \times \left[ -\frac{1}{2},\frac{1}{2} \right]\right) \frac{1}{2 C' \sqrt{n}} \\
&\geq \frac{1}{4C' \sqrt{n}}.
\end{split}
\end{equation*}
We have shown that there exists a constant $C = (4C')^{-1}$ and, for any $\ep >0$, there exists a constant $C_\ep>0$ such that we can bound the correlations by $C n^{-\frac{1}{2}} \leq \cov(\Phi \circ F^n, \psi) \leq C_\ep n^{-\frac{1}{2}+\ep}$, hence the bound of Theorem \ref{thm:poly_mixing_2} is, in this case, optimal.


\section{Skew-products over one-sided subshifts}\label{sec:one_side}

To prove Theorem \ref{thm:main1}, we have to first prove analogous statements for one-sided subshifts.
In this section, we discuss the case of skew-products over topologically mixing one-sided subshifts of finite type. 

Let $\sigma \colon X \to X$ be a topologically mixing one-sided subshift of finite type, equipped with a Gibbs measure $\mu = \mu_u$ with respect to the potential $u$. 
For  $0<\theta < 1$, the distance $d^+_\theta$ and the space of Lipschitz functions $\mathscr{F}^+_{\theta}$ are defined analogously to the case of the two-sided shift.
Let $f^+ \in \mathscr{F}^+_{\theta}$ be a real-valued Lipschitz function with zero average, and consider the skew-shift
\begin{equation}\label{eq:skew_product_one_sided}
F^+\colon X \times \R \to X \times \R, \ \ \ \ \ F(x, r) = (\sigma 
x, r + f^+(x)).
\end{equation}
Denote by $\nu$ the infinite measure $\mu \times \Leb$ on $X \times \R$. 
For any pair of global and local observables $\Phi, \psi$ over $X \times \R$, define the analogous correlation function
$$
\cov(\Phi \circ (F^+)^n, \psi) :=  \int_{X \times \R} ( \Phi \circ (F^+)^n) (x,r) \cdot \overline{\psi (x,r)} \diff \nu(x,r).
$$

\subsection{Good global and local observables for skew-shifts over one-sided subshifts}

The class of global and local observables we consider in this case are described below. 
In this setting, we require less regularity of the observables than in the case of two-sided shifts.

\begin{definition}[Good local observables -- one-sided case]\label{def:loc_obs_oneside}
Let $\mathscr{L}^+\subset L^1(\nu)$ be the space of functions $\psi \colon X \to \mathscr{S}$ such that, for every $\ell \in \N$, the function $x \mapsto \partial^{\ell} \psi(x)$ from $X$ to $L^1(\R)$ is 
Lipschitz. For every $\psi \in \mathscr{L}^+$, denote by $\Max_\ell(\psi)$ and $\Lip_\ell(\psi)$ the minimum constants such that
\begin{equation} \label{eq:constantbounds}
\norm{\partial^{\ell} \psi(x) }_{L^1(\R)} \leq \Max_\ell(\psi) 
\text{ and } 
\norm{\partial^{\ell}\psi(x) -\partial^{\ell}\psi(y)}_{L^1(\R)} \leq \Lip_\ell(\psi) d^+_{\theta}(x,y).
\end{equation}
\end{definition}
Let us remark that, if $\psi \in \mathscr{L}^+$, then, for every fixed $x \in X$, the Fourier transform $\widehat{\psi(x)}$ of $\psi(x) \in \mathscr{S}$ is a Schwarz function as well.
For any fixed $\xi \in \R$, we denote by $\widehat{\psi}_\xi \colon X \to \C$ the function $\widehat{\psi}_\xi(x)=\widehat{\psi(x)}(\xi)$.
\begin{lemma}\label{lemma:norm_theta}
Let $\psi \in \mathscr{L}^+$. For every $\xi \in \R$, we have $\widehat{\psi}_\xi \in \mathscr{F}^+_\theta$. Moreover, for every $\ell \geq 0$, and for all $\xi \neq 0$ we have
$$
\norm{\widehat{\psi}_\xi}_{\infty} \leq \Max_\ell(\psi)  \xi^{-\ell} \text{\ \ \ and\ \ \ } |\widehat{\psi}_\xi|_{\theta} \leq \Lip_\ell(\psi) \xi^{-\ell}.
$$
\end{lemma}
\begin{proof}
For any $\xi \neq 0$, $x \in X$, and $\ell \geq 0$ we have, by assumption 
in equation \eqref{eq:constantbounds},
$$
|\xi^{\ell}\widehat{\psi(x)} (\xi)| = |\widehat{\partial^{\ell}\psi(x)} (\xi)|\leq \norm{\widehat{\partial^{\ell}\psi(x)}}_{\infty} \leq \norm{\partial^{\ell}\psi(x)}_{L^1} \leq \Max_\ell(\psi),
$$
hence $\sup_x |\widehat{\psi}_\xi(x)| \leq \Max_\ell(\psi) \xi^{-\ell}$. 
Similarly, for any $x \neq y \in X$, 
\begin{equation*}
\begin{split}
|\xi^{\ell}[\widehat{\psi(x)} (\xi)- \widehat{\psi(y)} (\xi)]| &= |\widehat{\partial^{\ell}\psi(x)} (\xi) - \widehat{\partial^{\ell}\psi(y)} (\xi)|\leq \norm{\partial^{\ell}\psi(x) - \partial^{\ell}\psi(y)}_{L^1} \\
& \leq \Lip_\ell(\psi) d^+_{\theta}(x,y),
\end{split}
\end{equation*}
so that, for any fixed $\xi \neq 0$, we have $|\widehat{\psi}_\xi|_{\theta} \leq  \Lip_\ell(\psi) \xi^{-\ell}$.
\end{proof}
Let us recall that $\mathscr{A} \subset L^{\infty}$ denotes the space of Fourier-Stieltjes transforms of complex measures with finite total variation.
\begin{definition}[Good global observables -- one-sided case]\label{def:glob_obs_oneside}
Let $\mathscr{G}^+ \subset L^{\infty}(\nu)$ be the space of bounded functions $\Phi \colon X \to \mathscr{A}$. For $\Phi \in \mathscr{G}^+$, we define
$$
\| \Phi\|_{\mathscr{G}^+} := \sup_{x \in X} \norm{\eta_x}_{\TV},
$$
where, as usual, $\widehat{\eta_x} = \Phi(x)$.
\end{definition}
In the next sections, we will deal only with the non-invertible case of the skew-product $F^+$ and we will often suppress the $+$ in the notations introduced above, as it should not generate confusion. 
In Section \ref{sec:proof_thm_1} we will return to the invertible setting.

\subsection{Collapsed accessibility}

The property we need in the case of one-sided shifts which will replace the accessibility assumption is the following notion of \emph{collapsed accessibility}. 
\begin{definition} \label{def:collapsed_acc}
A Lipschitz function $f \colon X \to \R$ has the \emph{collapsed accessibility} property if there are constants $C$ and $N$ such that the following holds: for any $x \in X, t \in [0,1],$ and $n \geq 2 N$, there is a sequence of points
\[
    x_1, y_1, x_2, y_2, \ldots y_m, x_{m+1}
\]
such that
\begin{enumerate}
    \item $m \leq N$ and $x_1 = x_{m+1} = x$;
    \item $\sigma^n x_i = \sigma^n y_i$;
    \item $d(y_i, x_{i+1}) \le C r^n$; and
    \item $t = \sum_{k=1}^m f_n(x_k) - f_n(y_k)$.
\end{enumerate}
\end{definition}
The adjective ``collapsed'' refers to the fact that local stable manifolds
are collapsed to points when going from $\Sigma \times \R$ to $X \times \R.$

In order to prove Theorem \ref{thm:main1}, we will see in Section \ref{sec:proof_thm_1} that we can reduce an accessible skew-product $F$ to a skew-product $F^+$ over a one-sided shift such that $f^+$ enjoys the collapsed accessibility property.

\subsection{The one-sided version of the main theorem}

We state our main theorem in the case of skew-products over one-sided subshifts which have the collapsed accessibility property.
In Section \ref{sec:proof_thm_1}, we will deduce Theorem \ref{thm:main1} from Theorem \ref{thm:main2} below.
\begin{theorem}[Quantitative global-local mixing for one-sided subshifts]\label{thm:main2}
Assume that $f^+$, defined as in \eqref{eq:skew_product_one_sided}, has the collapsed accessibility property. Then, for every $\psi \in \mathscr{L}^+$, for every $\Phi \in \mathscr{G}^+$, for any $k \in \N$, and for every $\ep >0$, there exists a constant $C=C(\Phi, \psi, k, \ep)>0$ such that for every $n \in \N$, 
$$
|\cov(\Phi \circ F^n,\psi )| \leq C \left( \LF(\Phi, n^{-\frac{1}{2}+ \ep}) + n^{-k} \right).
$$
\end{theorem}

The \lq\lq low frequency\rq\rq\ term $\LF(\Phi, \cdot)$ in Theorem \ref{thm:main2} is defined exactly as in \eqref{defin:LF}, except that the integral is on $X$ instead of $\Sigma$.

\subsection{An expression for the correlation function}

The main tool to study the correlations is the \emph{transfer operator}. We recall the relevant definitions.

We denote by $L = L_\sigma \colon L^1(\mu) \to L^1(\mu)$ the transfer operator for the base dynamics $\sigma \colon X \to X$, namely the operator on $ L^1(\mu)$ defined implicitly by
$$
\int_{X} (v \circ \sigma) w \diff \mu = \int_{X} v \cdot (Lw) \diff \mu,
$$
for $v \in L^{\infty}(\mu)$ and $w \in  L^1(\mu)$.
Similarly, we denote by $L_{F^+}\colon L^1(\nu) \to L^1(\nu)$ the transfer operator associated to $F^+$, that is, the operator which, for every $\Phi \in L^{\infty}(\nu)$ and $\psi \in L^1(\nu)$, satisfies
$$
\int_{X\times \R} (\Phi \circ F^+) \psi \diff \nu = \int_{ X \times \R} \Phi \cdot (L_{F^+}\psi) \diff \nu. 
$$
Explicitly, for any $n \in \N$, we can write
\begin{equation}\label{eq:transfer_operator}
(L^nw)(x) = \sum_{\sigma^n y=x}e^{u_n(y)}w(y) 
\quad \textrm{ and } \quad
(L_{F^+}^n\psi) (x,r) = \sum_{\sigma^ny=x}e^{u_n(y)}\psi(y, 
r-f_n(y)).
\end{equation}
For any $z \in \C$, we let us further define the \emph{twisted transfer operator} ${\mathcal{L}}_{z} \colon L^1(\mu) \to L^1(\mu)$ by 
$$
(\mathcal{L}_{z}^nw)(x) = \sum_{\sigma^ny=x}e^{u_n(y) - i z f_n(y)} w(y),
$$
where $u$ is the potential defining the Gibbs measure and $u_n$ its cocycle.
Notice that all the operators described above restrict to operators acting on $\mathscr{F}^+_\theta$.

\begin{proposition}\label{thm:correlation_formula}
Let $\psi \in \mathscr{L}^+$ and $\Phi \in \mathscr{G}^+$. Then, for every $n \in \N$ we have
$$
\int_{X \times \R} ( \Phi \circ (F^+)^n) \cdot \overline{\psi} \diff \nu = \int_X \int_{-\infty}^\infty (\mathcal{L}_\xi^n \widehat{\psi}_\xi)(x) \diff \eta_x(\xi)\, \diff \mu(x).
$$
\end{proposition}
\begin{proof}
By definition of the transfer operator $L_{F^+}$, we can write
\begin{equation*}
\begin{split}
\int_{X \times \R} \Phi \circ (F^+)^n (x,r) \cdot \overline{\psi (x,r)} \diff \nu(x,r) &= \int_{X \times \R} \Phi (x,r) \cdot  L_{F^+}^n\overline{ \psi (x,r) }\diff \nu \\
&= \int_X \int_{-\infty}^\infty \Phi (x,r) \cdot  L_{F^+}^n\overline{ \psi (x,r) }\diff r \diff \mu,
\end{split}
\end{equation*}
where the applicability of the Fubini-Tonelli Theorem follows immediately from the definition of $\mathscr{G}^+$ and $\mathscr{L}^+$.
Since $\Phi(x)$ is the Fourier-Stieltjes transform of a measure $\eta_x$ we get 
\begin{equation*}
\begin{split}
\int_{X \times \R} ( \Phi \circ (F^+)^n) \cdot \overline{\psi} \diff \nu &= \int_X \int_{-\infty}^\infty \left( \int_{-\infty}^{\infty} e^{-ir\xi}  \diff \eta_x(\xi)  \right) \overline{L_{F^+}^n\psi (x,r)}\diff r \diff \mu \\
&= \int_X \int_{-\infty}^\infty \int_{-\infty}^{\infty} \overline{ e^{ir\xi} L_{F^+}^n\psi (x,r)} \diff \eta_x(\xi) \diff r \diff \mu.
\end{split}
\end{equation*}
For every $x \in X$, we have
\begin{equation*}
\begin{split}
\int_{-\infty}^\infty \int_{-\infty}^{\infty} |L_{F^+}^n\psi (x,r)| \diff |\eta_x|(\xi) \diff r & \leq \|\Phi\|_{\mathscr{G}^+} \int_{-\infty}^\infty |L_{F^+}^n\psi (x,r)| \diff r \leq \|\Phi\|_{\mathscr{G}^+} \|\psi(x)\|_1 \\
&\leq \|\Phi\|_{\mathscr{G}^+} \Max_0(\psi),
\end{split}
\end{equation*}
thus we can again apply the Fubini-Tonelli Theorem to get 
\begin{equation*}
\begin{split}
&\int_{X \times \R} ( \Phi \circ (F^+)^n) \cdot \overline{\psi} \diff \nu = \int_X \int_{-\infty}^\infty \overline{ \left( \int_{-\infty}^{\infty}  e^{ir\xi} L_{F^+}^n\psi (x,r) \diff r\right) } \diff \eta_x(\xi) \diff \mu \\
& \qquad = \int_X \int_{-\infty}^\infty \overline{  \widehat{L_{F^+}^n\psi} (x,-\xi) } \diff \eta_x(\xi) \diff \mu = \int_X \int_{-\infty}^\infty \widehat{L_{F^+}^n\psi} (x,\xi) \diff \eta_x(\xi) \diff \mu.
\end{split}
\end{equation*}
The conclusion follows by construction due to the equality
$$
\widehat{L_{F^+}^n\psi} (x,\xi) = (\mathcal{L}_{\xi}^n\psi_\xi)(x).
$$
\end{proof}


\section{Cancellations for twisted transfer operators}\label{sec:twisted_transfer_op}

From Proposition \ref{thm:correlation_formula}, it is clear that, in order to estimate the correlations, we need to study the twisted transfer operators $\mathcal{L}_\xi$, for real frequencies $\xi \in \R$. 
The aim of this section is to show that the collapsed accessibility property can be exploited to obtain some cancellations in the expression of $\mathcal{L}_\xi$.

Let us fix a complex-valued Lipschitz function $\tg \colon X \to \C$, and let $|\tg| = g$.
In this section, we use a tilde to denote a complex-valued or ``twisted'' function, and the same letter without a tilde to denote its absolute value.
We denote by $\mathcal{L} \colon \mathscr{F}_\theta^+ \to \mathscr{F}_\theta^+$ the operator defined by 
$$
(\mathcal{L} \tv)(x) = \sum_{\sigma y = x} \tg(y) \cdot \tv(y),
$$
and by $L\colon \mathscr{F}_\theta^+ \to \mathscr{F}_\theta^+$ the positive ``untwisted'' operator
$$
(L v)(x) = \sum_{\sigma y = x} g(y) \cdot v(y).
$$

Up to conjugating $L$ with a suitable multiplication operator, we can assume that $L 1 = 1$, namely
$$
\sum_{\sigma y = x} g(y) = 1,
$$
for all $x \in X$. 

Notice that the for the operator $\mathcal{L}_\xi$ defined in the previous section, we have $\tg = \exp(u + i \xi f)$, where $u$ is the potential for the Gibbs measure $\mu$.

One can easily see that $|\mathcal{L} \tv (x)| \leq Lv(x)$. Moreover, recall that $| \cdot |_\theta$ is the Lipschitz seminorm defined in \eqref{eq:def_Lip_seminorm}. Then, the following Lasota-Yorke inequality holds, see \cite[Proposition 2.1]{PaPo}.
\begin{lemma}[Basic inequality]\label{lemma:basic_inequality}
There exists a constant $C_0>0$ such that 
$$
|\mathcal{L} \tv|_{\theta} \leq \theta |\tv|_{\theta} + \tR \norm{\tv}_{\infty} ,
$$
where $\tR = C_0 |\tg|_\theta$.
\end{lemma}
By induction, 
$$
    (L^n v)(x) = \sum_{\sigma^n y = x} g_n(y) v(y)
    \text{\ \ \ and\ \  \ }
    (\mathcal{L}^n \tv)(x) = \sum_{\sigma^n y = x} \tg_n(y) \tv(y)
$$
where $g_n$ and $\tg_n$ are the cocycles
$$
    g_n(x) = g(x) g(\sigma(x)) \cdots g(\sigma^{n-1}(x))
    \text{\ \ \ and\ \ \ }
    \tg_n(x) = \tg(x) \tg(\sigma(x)) \cdots \tg(\sigma^{n-1}(x)).
$$
It follows that for all $n \geq 1$ we have
\begin{equation}\label{eq:basic_inequality_induction}
|\mathcal{L}^n \tv|_{\theta} \leq \theta^n |\tv|_{\theta} + \frac{\tR}{1-\theta} \norm{\tv}_{\infty}.
\end{equation}

\subsection{Collapsed accessibility and cancellation pairs}

Let us fix a positive constant $\ep > 0$, and an integer $n \geq 1$. 
We assume $\ep < \frac{1}{2}$ and $\ep < 1 - \theta$. 
Define 
$$
H := \max \left\{ 1, \frac{2 \tR}{1-\theta} \right\}.
$$

A Lipschitz function $\tv : X \to \C$ is a \emph{nice observable} if $|\tv|_\theta \leq H$ and $1 - \ep < v(x) < 1$ for all $x \in X$ (as always, $|\tv| = v$).

We say $\mathcal{L}$ has \emph{$(\ep,n)$-cancellation} if for any observable $\tv$ with $|\tv|_\theta \leq H$ and $0 \leq v(x) < 1$ for all $x \in X$, there is an integer $0 \leq k \leq n$ and a point $x \in X$ such that $ |\mathcal{L}^k \tv(x)| \leq 1 - \ep$.
We say $\mathcal{L}$ has \emph{strong $(\ep,n)$-cancellation} if for every nice observable $\tv$, there is a point $x \in X$ such that $ |\mathcal{L}^n \tv(x)| \leq 1 - \ep$. One can see that strong $(\ep,n)$-cancellation implies $(\ep,n)$-cancellation.

A pair of points $(x, y)$ in $X$ is a \emph{stable pair} if $\sigma^n x = \sigma^n y$.
We say a stable pair $(x, y)$ is a \emph{cancellation pair} for a nice observable $\tv$ if
\[
    |\tg_n(x) \tv(x) + \tg_n(y) \tg_n(y)| \leq g_n(x) v(x) + g_n(y) v(y) - \ep.
\]

\begin{lemma} \label{lemma:cancelpair}
If $(x, y)$ is a cancellation pair for $\tv$, then
$$
|\mathcal{L}^n \tv (p)| \leq 1 - \ep,
$$
where $p = \sigma^n x = \sigma^n y$.
\end{lemma}
\begin{proof}
By definition of cancellation pair, we have
\begin{equation*}
\begin{split}
|\mathcal{L}^n \tv (p)| & \leq  |\tg_n(x) \tv(x) + \tg_n(y) \tg_n(y)|  + \left\lvert  \sum_{\sigma^n q = p, \, q \neq x,y} \tg_n(q) \tv(q) \right\rvert \\
& \leq g_n(x) v(x) + g_n(y) v(y) - \ep + \sum_{\sigma^n q = p,\, q \neq x,y} g_n(q) v(q) \\
&\leq \left(  \sum_{\sigma^n q = p} g_n(q) \right) - \ep = 1-\ep,
\end{split}
\end{equation*}
where we used that $L^n1=1$.
\end{proof}
For a stable pair $(x, y)$ define the \emph{phase} of $(x, y)$ as
\[
    \arg \left( \frac{\tg_n(y)}{\tg_n(x)} \right).
\]
Here, arg is the complex argument and so the phase is the angle between $\tg_n(x)$ and $\tg_n(y)$ in the complex plane.
For the most part, we can just think of this value as an angle. 
However, if we include it in an inequality, we will assume it is a real number between $-\pi$ and $\pi$.

Define the \emph{stable tolerance} of $(x, y)$ as the number $0 < \delta < \pi$ which satisfies
\[
    1 - \cos(\delta) = \ep \left( \frac{1}{g_n(x)} + \frac{1}{g_n(y)} \right).
\]
Note that the right hand side must be less than two for this to be well defined.
In practice, we will always choose $\ep$ small enough so that this is the case.

\begin{proposition} \label{prop:stol}
Let $(x, y)$ be a stable pair, and $\tv$ a nice observable.
If $(x, y)$ is not a cancellation pair for $\tv$, then
    \[
        -\delta \leq
        \arg \left( \frac{\tg_n(x)}{\tg_n(y)} \frac{\tv(x)}{\tv(y)} \right) 
        \leq \delta,
    \]
where $\delta$ is the stable tolerance of $(x, y)$.
\end{proposition}
In other words, if $s$ is the phase of $(x, y),$ then
$$
    s + \delta < \arg\left( \frac{\tv(x)}{\tv(y)} \right) < s - \delta
$$
ignoring issues of the angle only being defined up to a multiple of $2\pi$.

To prove the proposition, we first establish the following lemma.

\begin{lemma} \label{lemma:zangle}
Let $z_1$ and $z_2$ be non-zero complex numbers with $\alpha = \arg(\frac{z_1}{z_2})$.
If
    \[
        \ep \left( \frac{1}{|z_1|} + \frac{1}{|z_2|} \right)
        \leq 2(1 - \cos(\alpha))
    \]
then
$$
        |z_1 + z_2| \le |z_1| + |z_2| - \ep.
$$
\end{lemma}

\begin{proof}
Write $z_0 = z_1 + z_2$ and $r_k = |z_k|$ for $k = 0,1,2$.
We wish to show that $r_0^2 \leq (r_1 + r_2 - \ep)^2$.
The cosine rule implies that
$$
        r_0^2 = r_1^2 + r_2^2 + 2 r_1 r_2 \cos(\alpha)
$$
and so it is enough to show 
$$
        2 r_1 r_2 \cos(\alpha) \le 2 r_1 r_2 - \ep(r_1 + r_2) + \ep^2.
$$
    This may be rewritten as
$$
        \ep \left( \frac{1}{r_1} + \frac{1}{r_2} \right)
        =
        \ep \frac{r_1 + r_2}{r_1 r_2}
        \le
        2 (1 - \cos(\alpha) ) + \frac{\ep^2}{r_1 r_2}.
        \qedhere
$$
\end{proof}        
\begin{proof}[Proof of Proposition \ref{prop:stol}.]
We show the contrapositive.
Define $z_1 = \tg_n(x) \tv(x)$ and $z_2 = \tg_n(y) \tv(y)$.
Then
$$
        |z_1| = g_n(x)v(x) \geq (1 - \ep) g_n(x),
$$
and a similar estimate holds for $|z_2|$.
Using $\ep < \frac{1}{2}$ and the definition of the stable tolerance, one sees that
    \[
        \ep \left( \frac{1}{|z_1|} + \frac{1}{|z_2|} \right)
        \leq 
        \frac{\ep}{1 - \ep} \left( \frac{1}{g_n(x)} + \frac{1}{g_n(y)} \right)
        \le
        2 (1 - \cos(\delta) ).
    \]
Let $\alpha$ be the angle between $z_1$ and $z_2$. 
If $\delta < \alpha$, then $1 - \cos(\delta) < 1 - \cos(\alpha)$ and Lemma \ref{lemma:zangle} shows that $(x, y)$ is a cancellation pair for $\tv$.
\end{proof}

For an arbitrary pair $(x, y)$ of points in $X$, define the \emph{unstable tolerance} as $0 \leq \delta < \frac{\pi}{2}$
such that
\[
    \sin(\delta) = 2 H d(x, y)
\]
Note that $x$ and $y$ must be reasonably close for this to be well defined.

\begin{proposition} \label{prop:utol}
If $(x, y)$ is a pair with unstable tolerance $\delta$ and $\tv$ is a nice observable, then
    \[    
        -\delta \leq \arg \left(\frac{\tv(x)}{\tv(y)} \right) \leq \delta.
    \] 
\end{proposition}
Again, we rely on a trigonometric lemma.

\begin{lemma} \label{lemma:sine}
Let $z_1$ and $z_2$ be non-zero complex numbers with angle $\alpha = \arg(\frac{z_1}{z_2})$.
If $0 < \alpha < \frac{\pi}{4}$ and $1 - \ep \leq |z_i| \leq 1$, then
$$
        (1 - \ep) \sin(\alpha) < |z_1 - z_2|.
$$
\end{lemma}    
\begin{proof}
Assume $|z_1| > |z_2|$ and consider the acute triangle defined by the points 0, $z_1$ and $z_2$ in complex plane.
Split this triangle into two right triangles by adding a line segment from $z_2$ to the opposite side of the triangle.
This new segment has length $|z_2| \sin(\alpha) \geq (1-\ep) \sin(\alpha)$ and so the line segment from $z_1$ to $z_2$ has length at least $(1 - \ep) \sin(\alpha)$.
\end{proof}
\begin{proof}
    [Proof of Proposition \ref{prop:utol}.]
Let $z_1 = \tv(x)$ and $z_2 = \tv(y)$ and let $\alpha$ be the angle between them.
The above lemma and the definition of ``nice'' together imply that
    \[
        (1 - \ep) \sin(\alpha) \le |z_1 - z_2| \le H d(x, y).
    \]
Since $\ep < \frac{1}{2}$ by assumption, the result follows.
\end{proof}    

A \emph{us-cycle} is a (finite) sequence of points in $X$:
\[
    x_1, y_1, x_2, y_2, \ldots, y_m, x_{m+1}
\]
where $x_1 = x_{m+1}$ and each pair $(x_k, y_k)$ is a stable pair.
The \emph{tolerance} of the cycle is the sum of the stable tolerances of the pairs
\[
    (x_1, y_1), \ (x_2, y_2), \ \ldots, \ (x_m, y_m)
\]
and the unstable tolerances of the pairs
\[
    (y_1, x_2), \ (y_2, x_3), \ \ldots, \ (y_m, x_{m+1}).
\]
We only consider us-cycles for which this tolerance is well defined.
The \emph{phase} of the cycle is
\[
    \arg \left(
    \frac{\tg(y_1)}{\tg(x_1)}
    \frac{\tg(y_2)}{\tg(x_2)}
    \cdots
    \frac{\tg(y_m)}{\tg(x_m)}
    \right).
\]
That is, the phase of the cycle is the sum of the phases of the individual stable pairs (up to a multiple of $2\pi$).
As defined, the phase is a number in $(-\pi, \pi]$.
We will only consider cycles where the phase is positive.

\begin{proposition} \label{prop:cancel}
If there is a us-cycle where the phase is greater than the tolerance, then $\mathcal{L}$ has strong $(\ep, n)$-cancellation.
\end{proposition}
\begin{proof}
Let $\tv$ be a nice observable.
Our goal is to show that one of the stable pairs in the cycle is a cancelling pair for $\tv$.
We assume none of them is a cancelling pair and derive a contradiction.
Let $S$ be the phase of the cycle and $\delta = \delta_s + \delta_u$ be the tolerance, where $\delta_s$ is the sum of the stable tolerances and $\delta_u$ is the sum of the unstable tolerances.
We are assuming $0 < \delta < S.$
Proposition \ref{prop:stol} implies that
    \[
        S - \delta_s <
        \arg \left(
        \frac{\tv(x_1)}{\tv(y_1)}
        \frac{\tv(x_2)}{\tv(y_2)}
        \cdots
        \frac{\tv(x_m)}{\tv(y_m)}
        \right)
        < S + \delta_s
    \]
    and Proposition \ref{prop:utol} implies that
    \[
        - \delta_u <
        \arg \left(
        \frac{\tv(x_2)}{\tv(y_1)}
        \frac{\tv(x_3)}{\tv(y_2)}
        \cdots
        \frac{\tv(x_{m+1})}{\tv(y_m)}
        \right)
        < \delta_u.
    \]
Since $x_1 = x_{m+1}$, the complicated product in the middle of each inequality is actually the same complex number and so we get $S - \delta_s < \delta_u$, a contradiction.
\end{proof}

\subsection{Cancellation by frequency}

We now apply the results above to the specific case of the operators $\mathcal{L}_\xi$ defined in the previous section, namely to the case 
$$
  (\mathcal{L}_\xi \tv)(x) = \sum_{\sigma y = x} \tg_\xi(y) \cdot \tv(x),
$$
where $\tg_\xi = \exp(u + i \xi f)$.
To simplify the presentation we only consider positive $\xi$, but analogous results will hold for negative frequencies.

One can show that the Lipschitz norm of $\tg_\xi$ satisfies $|\tg_\xi|_\theta \leq |g|_\theta + \xi |f|_\theta$,
and so each twisted operator satisfies a Lasota-Yorke inequality $ |\mathcal{L}_\xi \tv|_\theta \leq \theta |\tv|_\theta \ +  \tR_\xi \| \tv \|_\infty$ and
by Lemma \ref{lemma:basic_inequality}, $\tR=\tR_\xi$ grows linearly in $\xi$.

The notion of a ``nice observable'' will also depend on the frequency.
In particular, the value $H$ from the previous section depends on $\xi$ and so  $H = H_\xi = \frac{2}{1-\theta} \tR_\xi$, which also grows linearly in $\xi$.
Define a constant $G = \inf \{ \frac{1}{g(x)} : x \in X \}$ and an exponent $\alpha > 0$ determined by $\theta^\alpha G = 1$.
Note that $G$ and $\alpha$ are independent of the frequency. 

We now show that accessibility of the skew product leads to cancellation of these twisted operators and we give quantitative estimates of the amount of cancellation.

\begin{proposition}\label{prop:acccancel}
Suppose $f^+$ has the collapsed accessibilty property and $\xi_0 > 0$ is given.

Then there are positive constants $A$ and $B$ such that if $\xi \geq \xi_0$ and 
    \[
        \ep_\xi = \frac{1}{A G^{n_\xi}}
        \ \ 
        \text{where $n_\xi$ is the smallest integer which satisfies}
        \ \ 
        \theta^{n_\xi} < \frac{1}{B \xi},
    \]
    then $\mathcal{L}_\xi$ has strong $(\ep_\xi,n_\xi)$-cancellation.
\end{proposition}
\begin{remark}
One can see from the definitions of $\ep_\xi$ and $n_\xi$ that $\ep_\xi = \frac{1}{A} \theta^{\alpha n_\xi} \geq \frac{\theta}{A B^\alpha} \xi^{-\alpha}$.
\end{remark}
\begin{proof}
Assume without loss of generality that $0 < \xi_0 < \pi$.
The overall strategy of the proof is to use collapsed accessibility to show that, for any frequency $\xi \geq \xi_0$, there is a us-cycle with phase equal to $\xi_0$ and tolerance less than $\xi_0$.
Proposition \ref{prop:cancel} then gives cancellation.

Let $C$ and $N$ be as given in Definition \ref{def:collapsed_acc} of collapsed accessibility .
Then there is a constant $0 < a < 1$ such that any angle $0 < \delta < \pi$ which satisfies either $1 - \cos(\delta) \leq a$ or $\sin(\delta) \leq a$ also satisfies $\delta < \frac{1}{2N} \xi_0$.
Since $H_\xi$ grows linearly in $\xi$, there is $B > 0$ such that $2 C H_\xi \leq a B \xi$, for all $\xi \ge \xi_0$.
Up to increasing the value of $B$, we can also ensure that $n > 2N$ for any integer $n$ which satisfies $\theta^n < \frac{1}{B \xi_0}$.
Define $A = \frac{2}{a}$.
    
Now consider a specific frequency $\xi \ge \xi_0$ and use $n = n_\xi$ and $\ep = \ep_\xi$ defined as in the statement of the proposition.
Using this $n$ and $t = \frac{\xi}{\xi_0}$, there is a sequence of points $x_1, y_1, x_2, y_2, \ldots, y_m, x_{m+1}$ satisfying Definition \ref{def:collapsed_acc}. 
This sequence is a us-cycle for $\mathcal{L}_\xi$ and has phase equal to $\xi_0$.
If $\delta$ is the stable tolerance of a pair $(x_k, y_k),$ then
    \[
        1 - \cos(\delta) = \ep \left( \frac{1}{g_n(x_k)} + \frac{1}{g_n(y_k)} \right)
        \leq 2 \ep G^n
        = a.
    \]
If instead $\delta$ is the unstable tolerance of a pair $(y_k, x_{k+1}),$ then
    \[
        \sin(\delta) = 2 H_\xi d(y_k, x_{k+1})
        \leq 2 C H_\xi \theta^n
        \leq a B \xi \theta^n
        \leq a.
    \]
Together, these estimates show that the total tolerance of the us-cycle is less than $\xi_0$ and so Proposition \ref{prop:cancel} gives cancellation.
\end{proof}


\section{Contraction}\label{sec:contraction}

In this section, we show how to obtain some estimates on the norm of the operator $\mathcal{L}_\xi$. For high frequencies, we exploit the cancellations obtained in the previous section, while, for low frequencies, we apply some standard results from the perturbation theory of bounded linear operators. 

\subsection{High frequencies}

Recall that we defined $H := \max \left\{1,  \frac{2 \tR}{1-\theta} \right\}$.
It will be convenient to define the following norm on $\mathscr{F}^+_\theta$: let
$$
\|{\tilde v}\|_{H} := \max \left\{ \|{\tilde v}\|_\infty, \frac{| {\tilde v} |_\theta}{H}\right\}.
$$
Notice that the norms $\| \cdot \|_{H}$ and $\| \cdot \|_\theta$ are equivalent, namely
$$
\|{\tilde v}\|_{H} \leq \|{\tilde v}\|_{\theta} \leq 2H \|{\tilde v}\|_{H}. 
$$
In this section, we will prove the following result.
\begin{proposition}\label{corollary:norm_L_xi}
Suppose that $f^+$ has the collapsed accessibility property, and let $\xi_0>0$ be given. Then, there exists positive constants $A, B>0$ and an exponent $\beta > 0$ such that for all $\xi \geq \xi_0$ we have
$$
\|\mathcal{L}_\xi^{N}\|_{H} \leq 1- A \xi^{-\beta},
$$
for all $N \geq B |\log \xi|$.
\end{proposition}

We start by proving some simple preliminary results.
\begin{lemma} \label{lemma:Hbound}
    For any given $\xi >0$, if ${\tilde v} \in \mathscr{F}_\theta^+$, then $ \| \mathcal{L}_\xi {\tilde v} \| _H \leq \| {\tilde v} \| _H$.
\end{lemma}
\begin{proof}
Clearly, $\| \mathcal{L}_\xi {\tilde v} \|_\infty \leq \| {\tilde v} \|_\infty \leq \| {\tilde v} \|_H$. From the Basic Inequality in Lemma \ref{lemma:basic_inequality} we also get 
$$
\frac{|\mathcal{L}_\xi {\tilde v}|_\theta}{H} \leq \theta \frac{|{\tilde v}|_\theta}{H} + \frac{R}{H} \|{\tilde v}\|_\infty \leq \left( \theta + \frac{R}{H}  \right)  \| {\tilde v} \|_H \leq  \| {\tilde v} \|_H,
$$
since $H > {R}/(1-\theta)$. This completes the proof
\end{proof}

Let us recall that, from the definition of Gibbs measure, it follows that there exist constants $C_u, d$ such that for any ball $B(x,r)$ centered at $x \in X$ with radius $r \geq 0$ we can bound
\begin{equation}\label{eq:Gibbs}
\mu(B(x,r)) \geq C_u r^{d}.
\end{equation}

We will also use the fact that the untwisted transfer operator $L$ on $\mathscr{F}^+_\theta$ has a spectral gap, namely the following well-known result, see, e.g., \cite[Theorem 2.2]{PaPo}.
\begin{lemma}\label{lemma:LY}
There exist a bounded operator $\mathcal{N}\colon \mathscr{F}_{\theta}^+ \to \mathscr{F}_{\theta}^+$, a real number $0<\delta<1$, and a constant $C>0$ such that for all $n \in \N$ we have $\norm{\mathcal{N}^n}_{\theta} \leq C \delta^n$, and for all $\tv \in \mathscr{F}_{\theta}^+ $, 
$$
L^n(\tv) = \int_X \tv \diff \mu + \mathcal{N}^n(\tv).
$$
\end{lemma}

We have the following result.
\begin{lemma}\label{lemma:integral}
There exists a constant $C_1>0$ such that the following holds. 
For any $\varepsilon >0$, $\ell \geq 1$, and $\tilde v \in \mathscr{F}^+_\theta$ with $|{\tilde v}|_\theta \leq \ell$, if  $| {\tilde v} ({\bar x})| \leq 1-\varepsilon$ for a point ${\bar x} \in X$, then
$$
\int_X v \diff \mu \leq 1- C_1 \left( \frac{\varepsilon}{\ell} \right)^{d}\varepsilon.
$$
\end{lemma}
\begin{proof}
Define $w = 1-v$ and note that $w({\bar x}) \geq \varepsilon$ and the Lipschitz semi-norm of $w$ satisfies $|w|_\theta = |v|_\theta \leq |{\tilde v}|_\theta \leq \ell$.
If $B$ is the ball centered at ${\bar x}$ of radius $r = \frac{\varepsilon}{2 \ell}$, then $w(x) \geq \frac{\varepsilon}{2}$ for all $x \in B$, so
$$
\int_X w \diff \mu \geq \frac{\varepsilon}{2} \mu(B) \geq \frac{\varepsilon}{2} C_u \left( \frac{\varepsilon}{2\ell} \right)^{d}.
$$
Since $\int_X v \diff \mu = 1- \int_X w \diff \mu$, the result follows.
\end{proof}

\begin{lemma}\label{lemma:contraction}

There exist constants ${\bar A}, {\bar B} >0$ such that the following holds.
Assume that $\mathcal{L}= \mathcal{L}_\xi$ has $(\ep, n)$-cancellation. Then, for every ${\tilde v} \in \mathscr{F}_{\theta}^+$ with $\|{\tilde v}\|_{H} \leq 1$, and for any $N \geq N_0 := \lfloor -{\bar B} \log(\varepsilon/H) \rfloor$, we have
$$
\|\mathcal{L}^{N+n}{\tilde v}\|_{H} \leq 1-{\bar A} \frac{\varepsilon^{d +1}}{H^{d}}.
$$
\end{lemma}
\begin{proof}
Let $N_1$ and $N_2$ be the minimum integers which satisfy
$$
3C H \delta^{N_1} \leq \frac{C_1}{3} \frac{\varepsilon^{d +1}}{H^{d}} \text{\ \ \ and\ \ \ } 2 \theta^{N_2} \leq \frac{C_1}{3} \frac{\varepsilon^{d +1}}{H^{d}}, 
$$ 
and let $N_0 = N_1 + N_2$.
It is clear from the definition that there exists a constant ${ \bar B} >0$ such that $N_0 = \lfloor -{ \bar B} \log(\varepsilon/H) \rfloor$.

For every $x \in X$, 
$$
| \mathcal{L}^{N_1+n} {\tilde v} (x) | \leq |L^{N_1}( | \mathcal{L}^{n} {\tilde v}|) (x) | \leq \int_X | \mathcal{L}^n {\tilde v} | \diff \mu + |\mathcal{N}^{N_1}( | \mathcal{L}^{n} {\tilde v}|) (x) |, 
$$
where, by Lemma \ref{lemma:LY} and \eqref{eq:basic_inequality_induction},
$$
|\mathcal{N}^{N_1}( | \mathcal{L}^{n} {\tilde v}|) (x) | \leq C \delta^{N_1} \| \mathcal{L}^{n} {\tilde v} \|_\theta \leq C \delta^{N_1} (1 + \theta^n |{\tilde v}|_\theta +H) \leq 3C H \delta^{N_1} \leq \frac{C_1}{3} \frac{\varepsilon^{d +1}}{H^{d}},
$$
where the last inequality follows from the definition of $N_1$.
From Lemma \ref{lemma:Hbound}, we have $| \mathcal{L}^n {\tilde v} |_\theta \leq H \|{\tilde v}\|_H \leq H$. Thus, Lemma \ref{lemma:integral} applied to $| \mathcal{L}^{n} {\tilde v}|$ gives us
$$
| \mathcal{L}^{N_1+n} {\tilde v} (x) | \leq 1- C_1 \left( \frac{\varepsilon}{H} \right)^{d}\varepsilon +\frac{C_1}{3} \frac{\varepsilon^{d +1}}{H^{d}}.
$$
Therefore, we obtain
\begin{equation}\label{eq:inf_norm_N}
\| \mathcal{L}^{N+n} {\tilde v} \|_\infty \leq \| \mathcal{L}^{N_1+n} {\tilde v} \|_\infty \leq 1- \frac{2C_1}{3} \frac{\varepsilon^{d +1}}{H^{d}}.
\end{equation}
Finally, the inequality \eqref{eq:basic_inequality_induction} gives us
$$
| \mathcal{L}^{N+n} {\tilde v}|_\theta \leq \theta^{N_2} | \mathcal{L}^{N-N_2+n} {\tilde v}|_\theta + H \| \mathscr{L}^{N-N_2+n} {\tilde v} \|_\infty \leq H \left( 2 \theta^{N_2} +  \| \mathscr{L}^{N_1+n} {\tilde v} \|_\infty \right).
$$
By the definition of $N_2$ and \eqref{eq:inf_norm_N}, we conclude
$$
| \mathcal{L}^{N+n} {\tilde v}|_\theta \leq H \left( 2\theta^{N_2} +  \| \mathcal{L}^{N_1+n} {\tilde v} \|_\infty \right) \leq H \left( 1- \frac{C_1}{3} \frac{\varepsilon^{d +1}}{H^{d}}\right).
$$
The inequality above and \eqref{eq:inf_norm_N} conclude the proof.
\end{proof}

We are in position to complete the proof of Proposition \ref{corollary:norm_L_xi}
\begin{proof}[{Proof of Proposition \ref{corollary:norm_L_xi}}]
By Proposition \ref{prop:acccancel}, $\mathcal{L}_\xi$ has $(\ep_\xi, n_\xi)$-cancellations, with $\ep \geq A_0 \xi^{-\alpha}$ and $n_\xi \leq B_0 |\log \xi |$, for some positive constants $A_0, B_0$.
Therefore, by Lemma \ref{lemma:contraction}, for every $N \geq N_0 + n_\xi$ we have 
$$
\|\mathcal{L}_\xi^{N}\|_{H} \leq 1-{\bar A} \frac{\varepsilon^{d +1}}{H^{d}} \leq 1-A \xi^{-(\alpha d + \alpha +d)},
$$
for some constant $A>0$.
By the definitions of $N_0$ and $n_\xi$, there exists a constant $B>0$ such that $N_0 + n_\xi \leq B |\log \xi|$.
\end{proof}

\subsection{Low frequencies}

We now want to estimate the norm of $\mathcal{L}_\xi$ for small $\xi \in \R$. Let us notice that there exists $\xi_0 >0$ such that for all $ 0 \leq \xi \leq \xi_0$ we have $H = \max \left\{ 1, \frac{2 R}{1-\theta} \right\} = 1$, 
so that $\| \cdot \|_H \leq \| \cdot \|_\theta \leq 2 \|\cdot \|_H$.
We will prove the following bound.

\begin{proposition}\label{prop:norm_L_xi_low}
There exist $\kappa >0$ and a constant $A_\kappa >0$ such that, for all $0 < \xi < \kappa$ and for all $n \geq 0$, we have
$$
\| \mathcal{L}_\xi^n \|_H \leq 4 (1-A_\kappa \xi^2)^n.
$$
\end{proposition}

Let us recall that the family of operators $z \mapsto \mathcal{L}_z$ is \emph{analytic} for $z \in \C$. 
This ensures that we can apply classical results from analytic perturbation theory to study the spectrum of bounded linear operators, see in particular \cite[Theorem VII.1.8]{Kat}. 
In our case, since the operator $\mathcal{L}_0 = L$ has a spectral gap, we can deduce the following result, see \cite[Chapter 4]{PaPo}, \cite[Proposition 2.3]{Gou} and \cite[p.15]{Sar}.

\begin{theorem}\label{thm:analytic_perturbation}
There exists a $\kappa>0$ such that the twisted transfer operator $\mathcal{L}_z$ on $\mathscr{F}^+_\theta$ has a spectral gap for all $|z|<\kappa$. Moreover, there exist $\lambda_z \in \C$ and linear operators $\mathcal{P}_z$ and $\mathcal{N}_z$ such that $\mathcal{L}_z = \lambda_z \mathcal{P}_z + \mathcal{N}_z$ and which satisfy the following properties:
\begin{enumerate}
\item $\lambda_z$, $\mathcal{P}_z$ and $\mathcal{N}_z$ are analytic on the disk $\{|z| < \kappa\}$,
\item $\mathcal{P}_z$ is a projection and its range has dimension 1,
\item $\mathcal{P}_z\mathcal{N}_z = \mathcal{N}_z\mathcal{P}_z = 0$,
\item the spectral radius $\rho(\mathcal{N}_z)$ of $\mathcal{N}_z$ satisfies $\rho(\mathcal{N}_z) < \lambda_z - \delta$, for some $\delta$ independent of $z$. 
\end{enumerate} 
\end{theorem}

In our case, we restrict to real frequencies $0 < \xi < \kappa$. 
For the proof of the following lemma, see \cite[Chapter 4]{PaPo} and \cite[Section 4]{Sar}.
\begin{lemma}
With the notation of Theorem \ref{thm:analytic_perturbation}, there exist constants $A_\kappa, B_\kappa >0$ such that for all $0 < \xi < \kappa$ we have
$$
\left\lvert \lambda_\xi - (1 - 2A_\kappa \xi^2 ) \right\rvert \leq B_\kappa \xi^3.
$$
\end{lemma}
The fact that $A_\kappa$ is strictly positive follows from the fact that $f^+$ is not cohomologous to zero, see Lemma \ref{lemma:not_cohom_zero}.

\begin{proof}[{Proof of Proposition \ref{prop:norm_L_xi_low}}]
By Theorem \ref{thm:analytic_perturbation}, up to choosing a smaller $\kappa$, we can assume that $\rho(\mathscr{N}_\xi) < |\lambda_\xi| \leq 1- A_\kappa \xi^2$ for all $0 < \xi \leq \kappa$.
Therefore, for any $\tilde v \in \mathscr{F}^+_\theta$, we have
$$
\| \mathcal{L}_\xi^n {\tilde v} \|_H \leq \| \mathcal{L}_\xi^n {\tilde v} \|_\theta \leq (|\lambda_\xi|^n + \rho(\mathscr{N}_\xi)^n) \| {\tilde v} \|_\theta \leq 4(1-A_\kappa \xi^2)^n  \| {\tilde v} \|_H,
$$
which proves the result.
\end{proof}


\section{Rapid decay}\label{sec:decay}

In Section \ref{sec:pf_main2}, we will use the contraction results established for the twisted transfer operator $\mathcal{L}_\xi$ in the previous section in order to prove rapid mixing.
In this section, we give several technical propositions in an abstract setting which encapsulate most of the difficult inequalities involved in the proof.

\begin{definition} \label{def:rapiddecay}
Consider a function $w : A \subseteq (0,\infty) \to \R$.
We say $w(\xi)$ \emph{decays rapidly in} $\xi$ if for each $\ell \geq 1,$ there is a constant $C$ such that $|w(\xi)| \leq C \xi^{-\ell}$ for all $\xi$.
We say a sequence $\{s_n\}$ \emph{decays rapidly in} $n$ if for each $\ell \geq 1$, there is a constant $C$ such that $|s_n| \leq C n^{-\ell}$ for all $n$.
\end{definition}

\begin{proposition} \label{prop:rapid1}
    Suppose $A$, $B$, $\beta$, and $\xi_0$ are positive constants and that $\{w_n\}$ is a decreasing sequence of non-negative functions
    of the form $w_n \colon [\xi_0, \infty) \to [0, 1]$.
    If $w_0(\xi)$ decays rapidly in $\xi$ and
    \[
        w_{n+N}(\xi) \leq (1 - A \xi^{-\beta}) w_n(\xi)  \qquad \text{for all} \quad  N > B \log(\xi),
    \]
    then the sequence $\{s_n\}$ defined by $s_n = \sup_\xi w_n(\xi)$ decays rapidly in $n$.
\end{proposition}
In order to prove this, we first give a lemma which establishes for each fixed $\xi$ an exponential rate of decay of the sequence $\{w_n(\xi)\}$.

\begin{lemma} \label{lemma:halflife}
    In the setting of Proposition \ref{prop:rapid1}, there are constants $D$ and $\gamma$ such that $w_{n+K}(\xi) < \frac{1}{e} w_n(\xi)$ for all $K > D \xi^\gamma$.
\end{lemma}
\begin{proof}
    Consider a specific $\xi$ and let $k$ and $N$ be the smallest integers such that $k > \frac{1}{A} \xi^\beta$ and $N > B \log(\xi)$.
    Then $(1 - A \xi^{-\beta})^k < \exp(-k A \xi^{-\beta}) < \exp(-1)$ and so
    \[
        w_{n+ k N}(\xi) \leq (1 - A \xi^{-\beta})^k w_n(\xi) \leq \frac{1}{e} w_n(\xi).
    \]
    If we choose an exponent $\gamma > \beta$, then there is a constant $D$ such that
    \[
        k N \leq (\frac{1}{A} \xi^\beta + 1) (B \log \xi + 1) \le D \xi^\gamma.
    \]
    Moreover, this constant $D$ may be chosen uniformly for all $\xi$.
    If $K > D \xi^\gamma$, then $K > k N$ and $w_{n+K}(\xi) \leq w_{n + k N}(\xi)$.
\end{proof}
\begin{proof}
    [Proof of Proposition \ref{prop:rapid1}]
    We will show for any $q > 0$ that there is a constant $Q$ such that $s_n < \ep$ for all $0 < \ep < 1$ and $n > Q \ep^{-q}$.
    One can see that this condition implies that $\{s_n\}$ decays rapidly.
    Let $D$ and $\gamma$ be as in the above lemma and choose an integer $\ell > \gamma / q$.
    As $w_0(\xi)$ decays rapidly in $\xi$, there is $C$ such that $w_0(\xi) < C \xi^{-\ell}$ for all $\xi$ in the domain.
    For a given $\ep > 0$:
    \begin{itemize}
        \item let $a > 0$ be such that $C a^{-\ell} = \ep$,
        \item let $j$ be the smallest integer such that $e^{-j} < \ep$, and
        \item let $K$ be the smallest integer such that $K > D a^\gamma$.
    \end{itemize}
    Now consider a frequency $\xi$. If $\xi > a$, then $w_{j K}(\xi) \leq w_0(\xi) < \ep$.
    If instead $\xi \leq a$, then $K > D \xi^\gamma$ which implies $w_{j K}(\xi) \leq e^{-j} w_0(\xi) < \ep$.
    Together, these imply that $s_n < \ep$ for all $n > j K$.
    Since $a^\gamma = \frac{1}{C} \ep^{-\gamma/\ell}$ and $q > \gamma /\ell$, there is a constant $Q$ such that
    \[
        j K \leq ( \log(\ep^{-1}) + 1 )( D a^\gamma + 1 ) < Q \ep^{-q}
    \]
    holds uniformly for all $\ep$.
\end{proof}
\begin{proposition} \label{prop:lowdecay1}
    Suppose $A$, and $\xi_0$ are positive constants, $0 < \alpha < \frac{1}{2}$, and $\{w_n\}$ is a decreasing sequence of non-negative functions of the form $w_n \colon (0, \xi_0] \to [0, 1]$.
    If $w_{n+k}(\xi) \leq 4(1 - A \xi^{2})^k w_n(\xi)$ for all $\xi$, $k$ and $n$, then the sequence $\{s_n\}$ defined by
    \[
        s_n = \sup \ \{ w_n(\xi)  \ : \ n^{-\alpha} \leq \xi \leq \xi_0 \}
    \]
    decays rapidly in $n$.
\end{proposition}
\begin{proof}
    If $n^{-\alpha} \leq \xi$, then $w_n(\xi) \leq 4(1 - A \xi^2)^n \leq 4\exp(- n A \xi^2) \leq 4\exp(- n^{1-2\alpha} A)$, and, since $1 - 2\alpha > 0$, one can show that $4\exp(- n^{1-2\alpha} A)$ decays rapidly in $n$.
\end{proof}

Propositions \ref{prop:rapid1} and \ref{prop:lowdecay1} are enough to establish rapid mixing in the setting of skew products over one-sided shifts.
However, to handle two-sided shifts, we will need the following more technical results.

\begin{proposition}\label{prop:rapid2}
    Let $A$, $B$, $\beta$, $\xi_0$ and $\theta$ be positive constants.
    Let $v_{n,m} \colon [\xi_0, \infty) \to [0, \infty)$ be a collection of functions defined for all integers $n \geq 0$ and $m \geq 0$ and let $w \colon [\xi_0, \infty) \to [0, \infty)$ be a bounded function.
    Suppose for all $n,m \geq 0$ and $\xi \geq \xi_0$ that
    \begin{enumerate}
        \item $v_{n,m}(\xi) \leq v_{n+1,m}(\xi)$,
        \item $v_{n+N,m}(\xi) \leq (1 - A \xi^{-\beta})v_{n,m}(\xi)$ when $N > B \log(\xi)$,
        \item $v_{0,m}(\xi) \leq \theta^{-m} w(\xi)$, and
        \item $w(\xi)$ decays rapidly in $\xi$.
    \end{enumerate}
    Then, for any $c > 0,$ the sequence $\{t_n\}$ defined by
    \[    
        t_n = \sup \ \{ v_{n,m}(\xi) \ : \ \xi > \xi_0 \ \ \text{and} \ \ m < c \log(n) \}
    \]
    decays rapidly in $n$.
\end{proposition}        
\begin{proof}
    As $w(\xi)$ is bounded, we may without loss of generality assume that $w(\xi) \leq 1$ for all $\xi$.
    Define $w_n(\xi) = \sup_m \theta^m v_{n,m}(\xi)$.
    One can verify that $\{w_n\}$ satisfies the hypotheses of Proposition \ref{prop:rapid1}.
    Hence $\sup_\xi w_n(\xi)$ decays rapidly in $n$, meaning that for a given $\ell$, there is $C$ such that $v_{n,m}(\xi) < C \theta^{-m} n^{-\ell}$ for all $m$ and $n$.
    If $m < c \log(n)$, then
    \[  
    \theta^{-m} < \theta^{-c \log(n)} = n^{-c \log(\theta)} \quad \Rightarrow \quad v_{n,m}(\xi) < C n^{-\ell - c \log(\theta)}.
    \]
    From this, one can see that $\{t_n\}$ decays rapidly in $n$.
\end{proof}
\begin{proposition} \label{prop:lowdecay2}
    Let $A$, $\beta$, $\xi_0$ and $\theta$ be positive constants, and $0 < \alpha < \frac{1}{2}$.
    Let $v_{n,m} \colon (0, \xi_0] \to [0, \infty)$ be a collection of functions defined for all integers $n \geq 0$ and $m \geq 0$, and let $w \colon (0, \xi_0] \to [0, \infty)$ be a bounded function.
    Suppose for all $n,m,k \geq 0$ and $\xi \leq \xi_0$ that
    \begin{enumerate}
        \item $v_{n,m}(\xi) \leq v_{n+1,m}(\xi)$,
        \item $v_{n+k,m}(\xi) \leq 4(1 - A \xi^{2})^k v_{n,m}(\xi)$,
        \item $v_{0,m}(\xi) \leq \theta^{-m} w(\xi)$.
    \end{enumerate}
    Then for any $c > 0,$ the sequence $\{t_n\}$ defined by
    \[    
        t_n = \sup \ \{ v_{n,m}(\xi) \ : \  n^{-\alpha} < \xi \leq \xi_0 \ \ \text{and} \ \  m < c \log(n) \}
    \]
    decays rapidly in $n$.
\end{proposition}        
\begin{proof}
    This follows from Proposition \ref{prop:lowdecay1} using the proof of Proposition \ref{prop:rapid2}.
\end{proof}


\section{Proof of Theorem \ref{thm:main2}}\label{sec:pf_main2}

This section is devoted to the proof of Theorem \ref{thm:main2}. 

Let $\psi \in \mathscr{L}^+$ and $\Phi \in \mathscr{G}^+$ be given, and fix $k \in \N$ and $0 < \alpha < 1/2$.
Recall that the good global observable $\Phi$ defines a complex measure $\eta_x$ for each $x \in X$ and that there is a uniform constant $M = \|\Phi \|_{\mathscr{G}^+}$ such that $ \| \eta_x \| _{\TV} \leq M$ for all $x \in X$.
The Fourier transform of the good local observable $\psi$ is a function of the form $\widehat{\psi} \colon X \times \R \to \C$ where, for each frequency $\xi$, the function $\widehat{\psi}_\xi \colon X \to \C$ defined by $\widehat{\psi}_\xi(x) = \widehat{\psi(x)}(\xi)$ is H\"older and lies in $\mathscr{F}^+_\theta$.

By Proposition \ref{thm:correlation_formula}, we have
\begin{equation*}
\cov(\Phi \circ F^n,\psi) = \int_X \int_{-\infty}^\infty (\mathcal{L}_\xi^n \widehat{\psi}_\xi)(x) \diff \eta_x(\xi)\, \diff \mu(x) - \nuav(\Phi)\nu(\psi).  
\end{equation*}
We will estimate the correlations by splitting the frequencies $\xi \in \R$ into the cases $\xi = 0$, $0 < |\xi| < n^{-\alpha}$, and $|\xi| > n^{-\alpha}$.
In fact, we only consider $\xi \geq 0$ as the estimates for $\xi < 0$ are analogous.

The proof of Theorem \ref{thm:main2} follows from the next three lemmas.

\begin{lemma}\label{lemma:corr_zero_freq1}
There exist constants $C >0$ and $0< \delta <1$ (depending from $L$, as given in Lemma \ref{lemma:LY})  such that 
$$
\left\lvert \int_X \int_{\{0\}}(\mathcal{L}_\xi^n \widehat{\psi}_\xi)(x) \diff \eta_x(\xi)\, \diff \mu(x) - \nuav(\Phi)\nu(\psi) \right\lvert \leq C M (\Max_0(\psi) + \Lip_0(\psi) ) \delta^n,
$$
for all $n$.
\end{lemma}
\begin{proof}
Recalling that $\mathcal{L}_0 = L$ is the transfer operator associated to $\sigma$, we have
$$
\int_{\{0\}}(\mathcal{L}_\xi^n \widehat{\psi}_\xi)(x) \diff \eta_x(\xi)\, \diff \mu(x) = \eta_x(\{0\}) (L^n\widehat{\psi}_0)(x).
$$
By Lemma \ref{lemma:LY}, there exists a constant $C>0$ and $0<\delta<1$ such that 
$$
\left\lvert (L^n\widehat{\psi}_0)(x) - \int_X \widehat{\psi}_0(x) \diff \mu(x) \right\rvert = \left\lvert (L^n\widehat{\psi}_0)(x) - \nu(\psi)\right\rvert \leq C \delta^n \| \widehat{\psi}_0 \|_\theta,
$$
where we used that $\widehat{\psi}_0(x) = \int_\R \psi(x,r) \diff r$.
Hence, by Lemma \ref{lemma:nu_u} and Lemma \ref{lemma:norm_theta}, we conclude
\begin{equation*}
\begin{split}
&\left\lvert \int_X \int_{\{0\}}(\mathcal{L}_\xi^n \widehat{\psi}_\xi)(x) \diff \eta_x(\xi)\, \diff \mu(x) - \nuav(\Phi)\nu(\psi) \right\lvert \\
& \qquad \leq C \left(\int_X |\eta_x|(\{0\}) \diff \mu(x) \right) \| \widehat{\psi}_0 \|_\theta \delta^n  \leq C M (\Max_0(\psi) + \Lip_0(\psi) ) \delta^n.  
\end{split}
\end{equation*}
\end{proof}

\begin{lemma}\label{lemma:corr_low_freq1}
We have that 
$$
\left\lvert \int_X \int_{(0,n^{-\alpha})}(\mathcal{L}_\xi^n \widehat{\psi}_\xi)(x) \diff \eta_x(\xi)\, \diff \mu(x) \right\lvert \leq \Max_0(\psi) \LF(\Phi, n^{-\alpha}),
$$
for all $n$.
\end{lemma}
\begin{proof}
Since, by Lemma \ref{lemma:norm_theta}, we have $\|\mathcal{L}_\xi^n \widehat{\psi}_\xi \|_\infty \leq \|\widehat{\psi}_\xi \|_\infty \leq \Max_0(\psi)$, we obtain
\begin{equation*}
\begin{split}
&\left\lvert \int_X \int_{(0,n^{-\alpha})}(\mathcal{L}_\xi^n \widehat{\psi}_\xi)(x) \diff \eta_x(\xi)\, \diff \mu(x) \right\lvert \leq \Max_0(\psi) \int_X |\eta_x| \Big( (0, n^{-\alpha})\Big) \diff \mu(x) \\
& \qquad \leq \Max_0(\psi) \LF(\Phi, n^{-\alpha}),
\end{split}
\end{equation*}
which settles the proof.
\end{proof}

\begin{lemma} \label{lemma:corr_high_freq1}
    The sequence
    \[
        \left\{ \int_{X}
        \int_{[n^{-\alpha}, \infty)} \mathcal{L}_\xi^n \widehat{\psi}_\xi \, \diff \eta_{x} (\xi)\, \diff \mu(x) \right\}_{n \geq 0}
    \]
    decays rapidly in $n$.
\end{lemma}
\begin{proof}
    For each $n$, define a function $w_n \colon [0, \infty) \to [0, \infty)$ by
    \[
        w_n(\xi) = \| \mathcal{L}_\xi^n \widehat{\psi}_\xi \|_H.
    \]
    By Lemma \ref{lemma:Hbound}, $w_n$ is a decreasing sequence of functions, and by Lemma \ref{lemma:norm_theta}, $w_0(\xi)$ is a bounded function which (in the notation of Section \ref{sec:decay}) decays rapidly in $\xi$.
    Up to rescaling $\psi$, we may freely assume that $w_0$ takes values in $[0,1]$.

    Proposition \ref{prop:norm_L_xi_low} implies that $w_n$ restricted to $(0,\kappa]$ satisfies the hypotheses of Proposition \ref{prop:lowdecay1}. We then fix $\xi_0 = \kappa$, so that Proposition \ref{corollary:norm_L_xi} implies that $w_n$ restricted to $[\xi_0, \infty)$ satisfies the hypotheses of Proposition \ref{prop:rapid1}.
    Hence, the sequence defined by
    \[
        s_n = \sup \ \{ w_n(\xi) \ : \ n^{-\alpha} \leq \xi < \infty \}
    \]
    decays rapidly in $n$.
    Note that $ \| \mathcal{L}_\xi^n \widehat{\psi}_\xi \|_\infty \leq \| \mathcal{L}_\xi^n \widehat{\psi}_\xi \|_H$ and so, for each $n$,
    \[
        \left\lvert \int_{[n^{-\alpha}, \infty)} \mathcal{L}_\xi^n \widehat{\psi}_\xi \, \diff \eta_{x} (\xi)\right\rvert \leq \| \eta_x \|_{\TV} \, s_n
    \]
    As $\mu$ is a probability measure, it follows that
    \[
    \left\lvert \int_X \int_{[n^{-\alpha}, \infty)} \mathcal{L}_\xi^n \widehat{\psi}_\xi \, \diff \eta_{x} (\xi) \diff \mu(x) \right\rvert \leq M s_n,
    \]
    where $M$ is the uniform bound on $ \| \eta_x \|_{\TV}$.
\end{proof}


\section{From accessibility to collapsed accessibility}\label{sec:acc_collapsed_acc}

In this section, we relate the notion of accessibility for a skew-product $F$ as in \eqref{eq:skew_product} to the property of collapsed accessibility defined in Section \ref{sec:one_side}.

For a two sided shift $\sigma \colon \Sigma \to \Sigma$, let $X$ be the corresponding one sided shift and let $\pi \colon \Sigma \to X$ be the projection.
Note that $\pi$ is a continuous, surjective, open map.
We also write $x^+$ for $\pi(x).$

For $x \in \Sigma$, define $\Ws_0(x) = \pi^{-1} \pi(x)$. In other words, $y \in \Ws_0(x)$ if and only if $x^+ = y^+$.
For $n \in \Z$, define $\Ws_n(x) = \sigma^{-n} \Ws_0( \sigma^n x)$ and note that
\[
    \Ws_0(x) \subset \Ws_1(x) \subset \Ws_2(x) \subset \cdots 
\]
is an increasing sequence whose union is $\Ws(x)$.
For a subset $U \subset \Sigma$, write
\[
    \Ws_n(U) = \bigcup_{x \in U} \Ws_n(U).
\]
\begin{lemma} \label{lemma:wsnu}
    If $U \subset \Sigma$ is open, then $\Ws_n(U)$ is open for all $n$.
    If $K \subset \Sigma$ is compact, then $\Ws_n(K)$ is compact for all $n$.
\end{lemma}
\begin{proof}
    Since $\pi$ is an open map, $\Ws_0(U) = \pi^{-1} \pi(U)$ is an open set.
    Since $\sigma^n$ is a diffeomorphism $\Ws_n(U) = \sigma^{-n}(\Ws_0(\sigma^n U))$ is also open.
    A similar proof holds for compact sets.
\end{proof}
\begin{lemma} \label{lemma:wsndist}
    For points $x$ and $y$ in $\Sigma$ and $n \in \Z,$ the following are equivalent:
    \begin{enumerate}
        \item $y \in \Ws_n(x)$,
        \item $\dist(\sigma^{n+k} x, \sigma^{n+k} y) \leq \theta^k$ for all $k \geq 0$.
    \end{enumerate} 
    \end{lemma}
\begin{proof}
    One can show that each of these conditions is equivalent to the sequences of symbols for $x$ and $y$ satisfying $x_m = y_m$ for all $m \geq n$.
\end{proof}
Instead of projecting onto the future $x \mapsto x^+$, we can analogously project onto the past $x \mapsto x^-$.
Define local unstable manifolds by $y \in \Wu_0(x)$ if and only if $x^- = y^-$, and for $n \in \Z$ define $\Wu_n(x) = \sigma^n \Ws_0( \sigma^{-n} x)$.
Analogous versions of the above lemmas hold for these manifolds.

Let us now consider the skew-product \eqref{eq:skew_product}. 
Writing $p = (x,s)$ and $q = (y,t)$, we define local stable manifolds by $p \in \Ws_n(q)$ if and only if $p \in \Ws(q)$ and $x \in \Ws_n(y)$.
Define local unstable manifolds analogously.
For points $p$ and $q$ and an integer $n > 0$, a \emph{$us$-$N$-path} from $p$ to $q$ is a sequence
\[
    p = p_0, p_1, \ldots p_n = q
\]
such that $n \leq N$ and for each $0 \leq k < n$ either $p_{k+1} \in \Ws_N(p_k)$ or  $p_{k+1} \in \Wu_N(p_k)$.

For a point $p$, define $AC_N(x)$ by $q \in AC_N(x)$ if and only if there is a $us$-$N$-path from $p$ to $q$.
Note that $AC_N(x)$ form an increasing sequence whose union is $AC(p)$.

For a subset $U \subset \Sigma \times \R,$ define
\[
    AC_N(U) = \bigcup_{x \in U} AC_N(x).
\]
\begin{lemma} \label{lemma:acopen}
    If $U \subset \Sigma \times \R$ is open, then $AC_N(U)$ is open for all $n \geq 0$.
    If $K \subset \Sigma \times \R$ is compact, then $AC_N(K)$ is compact for all $n \geq 0$.
\end{lemma}
\begin{proof}
    This follows directly from Lemma \ref{lemma:wsnu}.
\end{proof}
\begin{proposition} \label{prop:acn}
    Let $K$ be a compact subset of $\Sigma \times \R$ such that $\overline{\ior(K)} = K$.
    If $p \in \Sigma \times \R$ is such that $K \subset AC(p)$, then there is $N$ such that $K \subset AC_N(p)$.
\end{proposition}
\begin{proof}
    Since $AC_N(p)$ is an increasing sequence of compact sets and $K$ is a Baire space, there is $N_1$ such that $AC_{N_1}(p)$ contains a non-empty open subset $U \subset K$.
    Since $AC_N(U)$ is an increasing sequence of open sets whose union contains the compact set $K$, there is $N_2$ such that $K \subset AC_{N_2}(U)$.
    Then $K \subset AC_{N_1+N_2}(p)$.
\end{proof}

We have the following result.
\begin{proposition}\label{prop:access_collapsed_access}
    Let $f \colon X \to \R$ be a Lipschitz function.
    If the skew product
    \[    
        F : \Sigma \times \R \to \Sigma \times \R, \quad
        (x, t) \mapsto (\sigma x, t + f(x^+) )
    \]
    is accessible, then $f$ has the collapsed accessibility property.
\end{proposition}

\begin{proof}
    Let $K = \Sigma \times [0,1]$.
    By Proposition \ref{prop:acn}, there is a uniform constant $N$ such that $AC_N(p)$ contains $K$ for any $p \in K$.
    With $N$ fixed, let $x \in X, t \in [0,1],$ and $n \geq 1$ be given.
    Then there is a sequence of points
    \begin{math}
        p_1, q_1, p_2, q_2, \ldots, q_m, p_{m+1}
    \end{math}
    such that
    \begin{enumerate} 
        \item $m \leq N, p_1 = F^{N-n}(x,0)$, and $p_{m+1} = F^{N-n}(x,t)$;
        \item $p_k \in \Ws_N(q_k)$; and
        \item $p_{k+1} \in \Wu_N(q_k)$.
    \end{enumerate}
    Applying $F^{N-n}$ to this sequence, we define $(a_k, s_k) = F^{N-n}(p_k)$ and $(b_k, t_k) = F^{N-n}(q_k)$ which satisfy $b_k \in \Ws_n(a_k)$ and $a_{k+1} \in \Wu_{2N-n}(b_k)$.
As $(a_k, s_k)$ and $(b_k, t_k)$ are on the same stable manifold in $\Sigma \times \R$, it follows that 
\begin{math}
        t_k - s_k = f_n(b_k^+) - f_n(a_k^+).
    \end{math}
By the unstable analogue of Lemma \ref{lemma:wsndist}, one can show that $d(b_k, a_{k+1}) \le \theta^{n-2n+1}$.
That is, $d(b_k, a_{k+1}) \leq C \theta^n$ where $C = \theta^{1-2N}$.
One can then check that $x_k = a_k^+$ and $y_k = b_k^+$ satisfy all of the conditions in the definition of collapsed accessibility.
\end{proof}


\section{Proof of Theorem \ref{thm:main1}}\label{sec:proof_thm_1}

We now prove Theorem \ref{thm:main1}. 
The strategy of the proof is to reduce the problem to the setting of Theorem \ref{thm:main2}. 

\subsection{Step 1: $f$ only depends on future coordinates}
Let us start with a preliminary step: we show that we can assume that the function $f$ in \eqref{eq:skew_product} only depends on the future coordinates. From \cite[Proposition 1.2]{PaPo}, we inherit the following result.
\begin{lemma}\label{lemma:f_cohomologous_f+}
There exist $h \in \mathscr{F}_{\sqrt{\theta}}$ and $f^+ \in \mathscr{F}^+_{\sqrt{\theta}}$ such that $f = f^+ + h - h\circ \sigma$.
\end{lemma}
When reducing to a one-sided shift, we will encounter some loss in regularity as in the previous lemma: the functions $h$ and $f^+$ are Holder with exponent $1/2$. We can however replace $\theta$ with $\sqrt{\theta}$ in the definition of the distance $d_\theta$ to make them Lipschitz. 
We remark that this is not an issue, and we will freely replace $\theta$ with a suitable choice that makes the functions Lipschitz. 

For any $\Phi \in \mathscr{G}$ and $\psi \in \mathscr{L}$, using Lemma \ref{lemma:f_cohomologous_f+}, we can write
\begin{equation*}
\begin{split}
\cov(\Phi \circ F^n,\psi) &= \int_\Sigma \int_{-\infty}^\infty \Phi(\sigma^n x, r+f_n(x)) \cdot \overline{\psi(x,r)} \diff r \diff \mu(x) \\
&= \int_\Sigma \int_{-\infty}^\infty \Phi(\sigma^n x, r+f^+_n(x) +h(x) -h(\sigma^nx)) \cdot \overline{\psi(x,r)} \diff r \diff \mu(x).
\end{split}
\end{equation*}
Let us define $\Phi_h(x,r) = \Phi(x, r-h(x))$ and $\psi_h(x,r)=\psi(x,r-h(x))$. We change variable $s=r+h(x)$ and we get
\begin{equation*}
\begin{split}
\cov(\Phi \circ F^n,\psi) &= \int_\Sigma \int_{-\infty}^\infty \Phi(\sigma^n x, s+f^+_n(x) -h(\sigma^nx)) \cdot \overline{\psi(x,s-h(x))} \diff s \diff \mu(x) \\
&= \int_\Sigma \int_{-\infty}^\infty (\Phi_h \circ F_1^n)(x,r) \cdot \overline{\psi_h(x,r)} \diff r \diff \mu(x),
\end{split}
\end{equation*}
where the skew-product $F_1$ is defined by $F(x,r) = (\sigma x, r + f^+(x))$. 
The map $H(x,r) = (x,r+h(x))$ used in the change of variable above is a conjugacy between $F$ and $F_1$, namely $H \circ F = F_1 \circ H$. Moreover, $H$ is uniformly continuous (more precisely, it is Lipschitz with respect to the distance $d_{\sqrt{\theta}}$, exactly as $h$), hence it preserves stable and unstable manifolds. In particular, $F_1$ is accessible.

The initial claim follows from the following lemma, whose proof is contained in the Appendix~\ref{sec:appendix2}.

\begin{lemma}\label{lemma:phi_h}
With the notation above, $\psi_h \in \mathscr{L}$ with $\nu(\psi_h) = \nu(\psi)$, and $\Phi_h \in \mathscr{G}$ with $\nuav(\Phi_h) = \nuav(\Phi)$. Moreover, for every $x \in \Sigma$, we have $|(\eta_h)_x| = |\eta_x|$, where $\widehat{(\eta_h)_x} = \Phi_h(x)$ and $\widehat{\eta_x}=\Phi(x)$.
\end{lemma}

\subsection{Step 2: observables only depend on future coordinates}

In the previous subsection, we have seen that we can assume that $f=f^+ \in \mathscr{F}^+_{\theta}$ (up to replacing $\theta$ with $\sqrt{\theta}$).
We now show that we can replace the observables $\Phi = \Phi_h \in \mathscr{G}$ and $\psi= \psi_h \in \mathscr{L}$ with observables in $\mathscr{G}^+$ and in $\mathscr{L}^+$ respectively: this is the content of Proposition \ref{lemma:one_sided} below.
The proof follows the same lines as in \cite{Dol}; however in our case there are some additional difficulties in showing that the functions defined belong to $\mathscr{G}^+$ and $\mathscr{L}^+$. In particular, we will need to use the assumption \eqref{eq:TC} to ensure some compactness property in $\mathscr{A}$. We postpone the proof to the Appendix~\ref{sec:appendix2}.
\begin{proposition}\label{lemma:one_sided}
Let $\Phi \in \mathscr{G}$ and $\psi \in \mathscr{L}$. 
There exist constants $K, M(\Phi) \geq 0$, sequences $\{\Phi_m\}_{m \in \N} \subset \mathscr{G}^+$, $\{\psi\}_{m \in \N} \subset \mathscr{L}^+$, and, for every $\ell \in \N$, there exist constants $M(\psi, \ell)$ and $L(\psi, \ell)$ such that the following properties hold for all $\ell, m, n \in \N$ and $x \in X$:
\begin{itemize}
\item[(i)] $\nuav(\Phi_m)=\nuav(\Phi)$ and $\nu(\psi_m) = \nu(\psi)$,
\item[(ii)] $\| \Phi \circ F^m(x, \cdot) - \Phi_m(x,\cdot) \| \leq M(\Phi) \theta^m$, and $\| \Phi(x)\|\leq M(\Phi)$,
\item[(iii)] $\Max_\ell(\psi_m) \leq M(\psi, \ell)$ and $\Lip_\ell(\psi_m) \leq \theta^{-m} L(\psi, \ell)$, 
\item[(iv)] $|\cov(\Phi \circ F^n,\psi) - \cov(\Phi_m \circ (F^+)^n,\psi_m)| \leq K \theta^m$.
\end{itemize}
\end{proposition}

From Proposition \ref{prop:access_collapsed_access}, it follows that the function $f^+$ in the definition of the one-sided skew-product $F^+$ has the collapsed accessibility property.

\subsection{Step 3: end of the proof}

We are now ready to prove Theorem \ref{thm:main1}.
Let $\Phi \in \mathscr{G}$ and $\psi \in \mathscr{L}$, and fix $k \in \N$ and $0<\alpha<1/2$.
Consider the sequence of functions $\{\psi\}_{m \in \N} \subset \mathscr{L}^+$ given by Proposition \ref{lemma:one_sided}. 
By Lemma \ref{lemma:norm_theta}, their Fourier transforms $(\widehat{\psi_m})_\xi$ satisfy
$$
\| (\widehat{\psi_m})_\xi \|_\infty \leq M(\psi, \ell) \xi^{-\ell} \text{\ \ \ and\ \ \ } |(\widehat{\psi_m})_\xi|_\theta \leq \theta^{-m} L(\psi, \ell) \xi^{-\ell}.
$$
If we define a function $w \colon (0, \infty) \to [0, \infty)$ by $w(\xi) = \sup_m \theta^{m}\| (\widehat{\psi_m})_\xi \|_H$, then these estimates imply that $w(\xi)$ decays rapidly in $\xi$ in the sense of Section \ref{sec:decay}.
We further define functions $v_{n,m}\colon (0, \infty) \to [0, \infty)$ by $v_{n,m}(\xi)=\| \mathcal{L}_\xi (\widehat{\psi_m})_\xi \|_H$, and we notice that $v_{n,m}$ and $w$ satisfy the hypotheses of Proposition \ref{prop:rapid2} and Proposition \ref{prop:lowdecay2}. 
Consequently, for any $c>0$, the sequence $\{t_n\}_{n \in \N}$ defined by
$$
t_n = \sup\ \{v_{n,m}(\xi)\ :\ n^{-\alpha} \leq \xi < \infty \text{\ \ \ and\ \ \ } m < c\log(n) \}
$$
decays rapidly in $n$.

Fix $n \in \N$ and let $m$ be the largest integer such that $m < c \log(n)$, where $c = k/(-\log(\theta))$; in particular
$$
n^{-k} = \theta^{c \log(n)} < \theta^m \leq \theta^{c \log(n)-1} = \theta^{-1} n^{-k}.
$$
By Proposition \ref{lemma:one_sided}-(iv), we get
$$
| \cov(\Phi \circ F^n,\psi) | \leq | \cov(\Phi_m \circ (F^+)^n,\psi_m)| + K \theta^{-1} n^{-k},
$$ 
hence it suffices to bound the first summand in the right-hand side above.
By Proposition \ref{thm:correlation_formula}, we have
\begin{equation*}
\begin{split}
| \cov(\Phi_m \circ (F^+)^n,\psi_m)| \leq & \left\lvert \int_X \int_{\{0\}}(\mathcal{L}_\xi^n (\widehat{\psi_m})_\xi)(x) \diff (\eta_m)_x(\xi)\, \diff \mu(x) - \nuav(\Phi_m)\nu(\psi_m) \right\lvert \\
&+ \left\lvert \int_X \int_{(0,n^{-\alpha})}(\mathcal{L}_\xi^n (\widehat{\psi_m})_\xi)(x) \diff (\eta_m)_x(\xi)\, \diff \mu(x) \right\lvert \\
&+ \left\lvert \int_X \int_{[n^{-\alpha}, \infty)}(\mathcal{L}_\xi^n (\widehat{\psi_m})_\xi)(x) \diff (\eta_m)_x(\xi)\, \diff \mu(x) \right\lvert.
\end{split}
\end{equation*}
The last summand in the right-hand side above is bounded by $M t_n$, hence decays rapidly. The first term, by Lemma \ref{lemma:corr_zero_freq1}, is bounded by
$$
C \| \Phi_m \|_{\mathscr{G}^+} (\Max_0(\psi_m) + \Lip_0(\psi_m) ) \delta^n \leq C M(\Phi) (M(\psi,0)+L(\psi,0)) n^k \delta^n,
$$
which decays rapidly as well. 
Finally, for the second term, Lemma \ref{lemma:corr_low_freq1} implies
$$
\left\lvert \int_X \int_{(0,n^{-\alpha})}(\mathcal{L}_\xi^n (\widehat{\psi_m})_\xi)(x) \diff (\eta_m)_x(\xi)\, \diff \mu(x) \right\lvert \leq M(\psi,0)  \LF(\Phi_m, n^{-\alpha}).
$$
In order to coclude the proof of Theorem \ref{thm:main1}, it suffices to establish the following lemma.

\begin{lemma}
With the notation above, for any $r>0$ we have
$$
|\LF(\Phi_m, r) - \LF(\Phi, r) | \leq M(\Phi) \theta^{-1} n^{-k}.
$$
\end{lemma}
\begin{proof}
Note that the measure associated to $(\Phi \circ F^m)(x,\cdot)$ is $e^{-i \xi f_m(x)} \diff \eta_{\sigma^m x}(\xi)$, whose variation is $|\eta_{\sigma^m x}|$.
Let us denote $R= (-r,r)\setminus \{0\} \subset \R$. Then, by Proposition \ref{lemma:one_sided}-(ii),
\begin{equation*}
\begin{split}
& |\LF(\Phi_m, r) - \LF(\Phi, r) | = \left\lvert \int_X |(\eta_m)_x|(R) \diff \mu(x) - \int_\Sigma |\eta_x|(R) \diff \mu(x)  \right\rvert \\
& \qquad \leq \left\lvert \int_\Sigma |\eta_{\sigma^m x}|(R) \diff \mu(x) - \int_\Sigma |\eta_x|(R) \diff \mu(x)  \right\rvert + \max_{x \in X} \|(\eta_m)_x - e^{-i \xi f_m(x)} \eta_{\sigma^m x} \|_{\TV} \\
& \qquad = \max_{x \in X} \|\Phi_m(x, \cdot) - \Phi \circ F^m(x,\cdot) \| \leq M(\Phi) \theta^m \leq M(\Phi) \theta^{-1} n^{-k}.
\end{split}
\end{equation*}
\end{proof}


\appendix


\section{Accessibility and symbolic dynamics} \label{sec:AnosovToSym}

Let $A \colon M \to M$ be a diffeomorphism and let $\Om \subset M$ be a transitive uniformly hyperbolic subset.
One can construct a Markov partition on $\Om$ and use it to define symbolic dynamics.
That is, there is a subshift of finite type $\sigma \colon \Sigma \to \Sigma$ and a continuous surjective map $\pi \colon \Sigma \to \Om$
such that $A \circ \pi = \pi \circ \sigma$.

Let $f \colon M \to \R$ be a continuous function which defines a skew product
\[
    F(x, t) = (A(x), t + f(x))
\]
on $M \times \R.$ This function then also defines a ``symbolic skew product''
\[
    \Fsym(x,t) = (\sigma(x), t + \fsym(x))
\]
on $\Sigma \times \R$ where $\fsym \colon \Sigma \to \R$ is given by $\fsym = f \circ \pi$.

Let $u \colon \Omega \to \R$ be a H{\" o}lder continuous function, and let $\mu = \mu_u$ be the unique equilibrium state for $u$. 
We want to obtain a quantitative mixing result for $F$ with respect to $\nu = \mu \times \Leb$ by applying Theorem \ref{thm:main1} to the symbolic skew-product $\Fsym$.
In this appendix, we discuss how the classes of good local and global observables and the accessibility property translate from the original system to the symbolic counterpart.

\subsection{The observables}

The classes of good local and global observables on $M \times \R$ we consider are defined, repsectively, as H{\" o}lder functions 
$\psi \colon M \to \mathscr{S}$ from $M$ to the space of Schwartz functions $\mathscr{S}$ and H{\" o}lder functions $\Phi \colon M \to \mathscr{A}$ from $M$ to the Fourier-Stieltjes algebra $ \mathscr{A}$ such that the tightness condition \eqref{eq:TC} is satisfied (with $x \in \Sigma$ replaced by $p \in M$).
If we equip the symbolic system with the invariant measure $\nu_{\text{sym}}= \mu_{u \circ \pi} \times \Leb$, where $\mu_{u \circ \pi}$ is the Gibbs measure with potential $u \circ \pi$, we obtain the following lemma.
\begin{lemma}
If $\psi$ is a good local observable on $M \times \R$, the function $\psi_{\text{sym}} := \psi \circ \pi$ is a good local observable for the symbolic system.
Similarly, if $\Phi$ is a good global observable on $M \times \R$, then $\Phi_{\text{sym}}:= \Phi \circ \pi$ is a good global observable for the symbolic system. Moreover,
$$
\cov(\Phi \circ F^n, \psi) = \cov(\Phi_{\text{sym}} \circ \Fsym^n, \psi_{\text{sym}}),
$$
where the reference measure is $\nu$ on the left and $\nu_{\text{sym}}$ on the right hand-side. 
\end{lemma}
\begin{proof}
For any choice of $\theta \in (0,1)$, the semiconjugacy $\pi \colon \Sigma \to \Omega$ is H{\" o}lder continuous with respect to the distance $d_{\theta}$ on $\Sigma$ (see, e.g., \cite[Lemma 4.2]{BowenBook}). 
Hence, given any $\psi$ and $\Phi$ as in the statement, the functions $\Phi_{\text{sym}}$ and $\psi_{\text{sym}}$ are H{\" o}lder continuous and, up to taking a larger $\theta$, they are actually Lipschitz. This shows that they are good local and global observables for the symbolic system. 

By \cite[Theorem 4.1]{BowenBook}, the equilibrium state $\mu= \mu_u$ can be expressed as the push-forward $\mu_u = \pi^{\ast} \mu_{u \circ \pi}$. The final claim follows immediately from the semiconjugacy between $\Fsym$  and $A$. 
\end{proof}

\subsection{Accessibility}

In order to apply Theorem \ref{thm:main1} to $\Fsym$, one needs to address the following question.

\begin{question}
    If $F|_{\Om \times \R}$ is accessible, does it follow that $\Fsym$ is accessible?
\end{question}
At first glance, the question might seem easy to answer, but there are some subtle issues here.
As $\pi$ is uniformly continuous, if points $x$ and $y$ lie on the same stable manifold in $\Sigma$, then they project down to points $\pi(x)$ and $\pi(y)$ lying on the same stable manifold in $M$.
However, one can construct examples where $\pi(x)$ and $\pi(y)$ lie on the same stable manifold, but $\sigma^n(x)$ and $\sigma^n(y)$ stay far apart for all $n \in \Z$.
Hence, not all $us$-paths in $M$ lift to $us$-paths in $\Sigma$.
Despite this, we can establish accessibility of $\Fsym$ in certain settings.
We will also discuss the difficulties involved in the general case later in
this appendix.

We will adopt the notation used in Bowen's book on the subject \cite{BowenBook}.
In particular recall that if $p$ and $q$ are in the same rectangle, then $[p,q]$ is the intersection of the local stable manifold of $p$ with the local unstable manifold of $q$.
If $x = \{x_n\}$ and $y = \{y_n \}$ are elements of the symbolic dynamics with the same ``zeroth'' symbol $x_0 = y_0,$ then $[x, y] = z = \{ z_n \}$ is defined by $z_n = x_n$ for $n \geq 0$ and $z_n = y_n$ for $n \leq 0$ and one can show that $\pi([x, y]) = [\pi(x), \pi(y)]$.

For $p$ and $q$ in $\Om$, if  $q \in W^s_A(p)$, define $\Delta^s(p,q) = \sum_{n=0}^\infty f(A^n q) - f(A^n p)$ and note that $(p,s) \in W^s_F(q,t)$ if and only if $t - s = \Delta^s(p,q)$.
If $q \in W^u_A(p),$ define $\Delta^u(p,q)$ analogously.
For $p$ and $q$ in the same rectangle, define
\[
    h(p,q) = \Delta^s(p, [p,q])
    + \Delta^u([p,q], q)
    + \Delta^s(q, [q,p])
    + \Delta^u([q,p], p).
\]
That is, $h(p,q)$ measures the height of the ``Brin quadrilateral''
that has $p$ and $q$ as two of its four vectices.
Note that $h$ is continuous and if $q$ is on the local stable or manifold manifold of $p$, then $h(p,q) = 0$.

Define $\Delta^s_{\text{sym}}, \Delta^u_{\text{sym}},$ and $\hsym$ using the same formulas, but with $\fsym$ in the place of $f$.
Then on the cylinder $C_i \subset \Sigma$ consisting of the sequences whose zeroth symbol corresponds to $R_i$, the function $\hsym \colon C_i \times C_i \to \R$ is continuous and $\hsym(x,y) = h(\pi(x), \pi(y))$.     

\begin{proposition}\label{prop:acccurve}
    If $p \in R_i$ and $\gamma \colon [0,1] \to R_i$ is a continuous curve such that $h(p, \gamma(0))$ is zero and $h(p, \gamma(1))$ is non-zero, then $\Fsym$ is accessible.
\end{proposition}
\begin{proof}
    Since $h$ and $\gamma$ are continuous, $I = h(p, \gamma([0,1]))$ is a positive length interval containing zero.
    For any $t \in [0,1],$ using the properties of symbolic dynamics we may find elements $x, y \in C_i$ such that $\pi(x) = p$ and $\pi(y) = \gamma(t)$.
    This implies that $\hsym(C_i \times C_i)$ contains $I$.
    From this, one can show that $\Fsym$ has an open accessibility class and then use this to conclude that $\Fsym$ is accessible.
\end{proof}    
\begin{corollary}\label{cor:genacc}
    If $A \colon M \to M$ is an Anosov diffeomorphism and $\Om = M$, then for a $\mathscr{C}^1$-open and $\mathscr{C}^r$-dense ($1 \leq r < \infty$) set of choices of $f \colon M \to \R$, the corresponding $\Fsym$ is accessible.
\end{corollary}
\begin{proof}
    Choose a rectangle $R_i$ in a partition.
    Since $A$ is Anosov and $R_i$ has interior, there is a periodic point $p$ in the interior of $R_i$.
    Adapting the proof of the ``Unweaving Lemma'', that is \cite[Lemma A.4.3]{HHU},
    we can make a small perturbation to any starting $f$ in a small neighbourhood in order to find a point $q$ with $h(p,q) \neq 0$.
    We can then define a path $\gamma \colon [0,1] \to R_i$ from $p$ to $q$ and apply Proposition \ref{prop:acccurve}.
\end{proof}    
\begin{corollary}
    If $A \colon \T^2 \to \T^2$ is an Anosov diffeomorphism, then $F$ is accessible if and only if $\Fsym$ is accessible.
\end{corollary}
\begin{proof}
    Here, $A$ is topologically conjugate to a linear map \cite{Fra}, and so we assume $A$ itself is linear. In this setting, we can find a Markov partition where the interior of each rectangle is homeomorphic to a disc and its boundary consists of two stable curves and two unstable curves.
    For such a construction, see for instance \cite[Section 8]{Roy}.
    If $\Fsym$ is not accessible, then for each rectangle $R_i$ the stable and unstable directions of $F$ are jointly integrable inside the region $R_i \times \R \subset M \times \R$. This region is therefore foliated by $\mathscr{C}^1$ surfaces (with boundary) tangent to $\Es \oplus \Eu$.
    As the rectangles meet along stable and unstable curves, we can ``glue together'' the leaves of neighbouring rectangles to produce a foliation on all of $M \times \R$ tangent to $\Es \oplus \Eu$.
\end{proof}
\begin{remark}
    The proof above is specific to the 2-torus.
    Bowen showed in higher dimensions that the boundaries of the rectangles are not smooth \cite{Bow78}.
There are several slightly different constructions of Markov partitions
\cite{Bowen70, Bowen73, BowenBook, Rat}.
In none of these contructions is connectedness of the rectangles
discussed and it is not clear from the constructions if a rectangle
is connected when the diffeomorphism is Anosov.
Even if the construction begins with connected rectangles,
complicated topological manipulations are necessary to
turn this into a family rectangles satisfying all of the axioms of
a Markov partition.
It is not clear how this steps affect connectedness.

It may conceivably be the case that the interior of a rectangle $R_i$
consists of infinitely many connected components.
The union $U = \bigcup_i \ior R_i$ of the interiors of all of the rectangles
would then be an open and dense set with infinitely many connected components.
Our reasoning above shows that, if $\Fsym$ is not accessible,
then the height function $h(x,y)$ is constant when $x$ and $y$ 
are in the same connected component of $U.$
Suppose a small neighbourhood $V$ of a point $p \in M$ intersects infinitely many
connected components of $U.$
Then the function
$V \to R, \ q \mapsto h(p,q)$ may be a Cantor function or a ``Devil's
staircase''
which is constant on an open and dense subset of $V$ despite not being constant
on all of $V.$

Here, the regularity of the stable and unstable foliations
is important, as Cantor functions are possible with H\"older regularity,
but not with Lipschitz regularity.
Keeping $p$ as above fixed, consider a point $q$ near $p.$
We analyze what happens as we move $q$ along its unstable manifold
until it intersects the stable manifold of $p.$
In other words, we take $q$ tending towards $[p,q]$ while
holding both $p$ and $[p,q]$ constant.
In the formula defining $h(p,q),$ the term
$\Delta^u(p,[p,q])$ will then be constant.
The term $\Delta^u([p,q],q)$
will depend in a $\mathscr{C}^1$ manner on the point $q$
since we are a moving along a single unstable manifold
and this single manifold is $\mathscr{C}^1$ regular.
The term $\Delta^u([q,p],p)$
also depends in a $\mathscr{C}^1$ manner on the point $q$
for the same reason.
The problematic term is $\Delta^s(q,[q,p]).$
As $q$ moves along its path toward $[p,q]$,
it passes through different stable leaves.
As a consequence,
$\Delta^s(q,[q,p]),$ which measures the change in ``height''
of stable segments on $M \times \R,$
is only as regular as the stable foliation on $M \times \R.$

In general, the stable foliation is only H\"older continuous
\cite{HPS} and so we cannot make further conclusions.
In certain settings, however, we know that the stable foliation
is $\mathscr{C}^1$ and  we can conclude that $\Delta^s(q,[q,p])$
is $\mathscr{C}^1$ and that its derivative is unifomly bounded.
From this one can show that the function $h$ is Lipschitz
in the sense that $h(x,y)$ is bounded by a constant times $d(x,y).$
The pathology of Cantor functions may therefore be ruled out.
Further, the $\mathscr{C}^1$ nature of the stable foliations allows
one to establish a property known as ``uniform non-integrably''
or UNI
which is a quantitative form of accessibility and can be used
to establish exponential mixing.
This is, more or less, the argument applied in
\cite{ArMe, ABV, BuWa}
where they establish exponential mixing 
in a setting where stable foliation is $\mathscr{C}^1.$

\end{remark}
We now consider accessibility in the setting of hyperbolic attractors.

\begin{corollary}
    If $A \colon M \to M$ is an Axiom A diffeomorphism and $\Om \subset M$ is an attractor or repellor, then for a $\mathscr{C}^1$-open and $\mathscr{C}^r$-dense ($1 \leq r < \infty$) set of choices of $f \colon M \to \R$, the corresponding $\Fsym$ is accessible.
\end{corollary}
\begin{proof}
    We will assume $\Om$ is an attractor so that for each $p \in \Om$, the unstable manifold $W^u_A(p)$ is a connected immersed submanifold lying entirely within $\Om$.

    As shown in \cite{HPPS}, there is a neighbourhood of $\Om$ on which one can define invariant stable and unstable foliations.
    We may find a periodic point $p$ and a small neighbourhood $U$ of $p$, such that every unstable manifold intersects $U$ in a
    path-connected set.

    As in the proof of Corollary \ref{cor:genacc} above, one can adapt the ``Unweaving Lemma'' to perturb $f$ and find $p$ and $q$ with $h(p,q) \neq 0$. We define $\gamma \colon [0,1] \to R_i$ to be a path from $[p,q]$ to $q$ and then apply Proposition \ref{prop:acccurve}.
\end{proof}


\section{Proofs of Lemmas \ref{lemma:not_cohom_zero}, \ref{lemma_app_1a}, \ref{lemma:nu_u}, and \ref{lemma:appendix3}}\label{sec:proofs_of_lemmas}

\subsection{Proof of Lemma \ref{lemma:not_cohom_zero}}
Assume that $f$ is cohomologous to zero, namely there exists a measurable function $w$ such that $f= w \circ \sigma - w$. By Liv\u{s}ic Theorem, we can assume that $w$ is continuous. We claim that for any $x \in \Sigma$, all the points that can be reached by an $su$-path from $(x, w(x)) \in \Sigma \times \R$ are contained in the graph of $w$, i.e.~in the set $\mathcal{G}(w):= \{ (y,w(y)): y \in \Sigma\}$, which will give a contradiction with the accessibility assumption.

Let $(x,w(x)) \in \mathcal{G}(w)$; we now show that the whole stable set $W^s(x,w(x))$ is fully contained in $\mathcal{G}(w)$. Let $(y,s) \in W^s(x,w(x))$; then, $y \in W^s(x)$ and 
$$
s-w(x) = \lim_{n \to \infty} f_n(x) - f_n(y) = \lim_{n \to \infty} (w (\sigma^nx) - w(x) - w(\sigma^ny) +w(y) ) = -w(x)+w(y),
$$
that is $s=w(y)$. This implies $W^s(x,w(x)) \subset \mathcal{G}(w)$. An analogous argument shows that $W^u(x,w(x)) \subset \mathcal{G}(w)$.

\subsection{Proof of Lemma \ref{lemma_app_1a}}
For any fixed $R>0$, by the invariance of the Gibbs Measure with respect to the dynamics, we have 
$$
\frac{1}{2R} \int_{\Sigma \times [-R,R]} \Phi(x,r) \diff \nu(x,r) = \frac{1}{2R} \int_{\Sigma \times [-R,R]} \Phi(\sigma x,r) \diff \nu(x,r),
$$
therefore
\begin{equation*}
\begin{split}
&\frac{1}{2R} \left| \int_{\Sigma \times [-R,R]} ( \Phi \circ F - \Phi) (x,r) \diff \nu(x,r) \right| \\
&\qquad \leq \frac{1}{2R}\left| \int_{\Sigma} \left( \int_{-R+f(x)}^{R+f(x)} \Phi(\sigma x,r) \diff r \right) \diff \mu(x) - \int_{\Sigma \times [-R,R]} \Phi(\sigma x, r) \diff \nu(x,r) \right| \\
&\qquad \leq \frac{1}{R} \norm{f}_{\infty} \norm{\Phi}_{\infty}.
\end{split}
\end{equation*}

The last term above converges to zero for $R \to \infty$, hence the limit
$$
\lim_{R \to \infty} \frac{1}{2R}  \int_{\Sigma \times [-R,R]} \Phi \circ F (x,r) \diff \nu(x,r)
$$
exists and equals $\nuav(\Phi)$.

\subsection{Proof of Lemma \ref{lemma:nu_u}}

By definition, we can write
$$
\Phi(x,r) = \int_{-\infty}^{\infty} e^{-ir\xi} \diff \eta_x(\xi) = \eta_x(\{0\}) + \int_{\R \setminus \{0\}} e^{-ir\xi} \diff \eta_x(\xi).
$$
We have to show that the limit \eqref{eq:lim_defin} exists; in order to do this, we prove that for any $x \in \Sigma$ we have
\begin{equation}\label{eq:claim_lemma_nu_u}
\lim_{R \to \infty} \frac{1}{2R} \int_{-R}^R  \Phi(x,r) \diff r = \eta_x(\{0\}),
\end{equation}
from which the claim follows.
For any $R >0$, 
\begin{equation*}
\begin{split}
&\frac{1}{2R} \int_{-R}^R \Phi(x,r) \diff r =  \frac{1}{2R} \int_{-R}^R \left( \eta_x(\{0\}) + \int_{\R \setminus \{0\}} e^{-ir\xi} \diff \eta_x(\xi) \right) \diff r \\
&= \eta_x(\{0\}) + \int_{\R \setminus \{0\}} \left( \frac{1}{2R} \int_{-R}^R  e^{-ir\xi} \diff r \right) \diff \eta_x(\xi) \\
&= \eta_x(\{0\}) + \int_{\R \setminus \{0\}} \frac{e^{iR\xi}-e^{-iR\xi}}{2iR\xi} \diff \eta_x(\xi) =  \eta_x(\{0\}) + \int_{\R \setminus \{0\}}  \frac{\sin(R\xi)}{R\xi} \diff \eta_x(\xi).
\end{split}
\end{equation*}
We have that $|\sin(R\xi)/(R\xi)|\leq 1$ (and $1 \in L^1( |\eta_x| )$, since the total variation of $\eta_x$ is finite) and $\sin(R\xi)/(R\xi)$ converges to $0$ for all $\xi\neq 0$. Lebesgue theorem yields \eqref{eq:claim_lemma_nu_u}, which completes the proof.

\subsection{Proof of Lemma \ref{lemma:appendix3}}

Let us show that for all $\ep>0$ we have
$$
\eta \left(\R \setminus  \left[-\frac{2}{\ep},\frac{2}{\ep} \right] \right) \leq \frac{1}{\ep} \int_{-\ep}^{\ep} \Phi(0) - \Phi(r) \diff r
$$
Then, simply by choosing $\ep = 2/K$, we conclude
$$
\eta(\R \setminus [-K,K]) \leq \frac{K}{2} \int_{-2/K}^{2/K} |\Phi(0)-\Phi(r)| \diff r \leq \frac{LK}{2} \int_{-2/K}^{2/K} |r| \diff r = \frac{2L}{K}.
$$

Fix $\ep>0$; then
\begin{equation*}
\begin{split}
& \frac{1}{2\ep} \int_{-\ep}^{\ep} \Phi(0) - \Phi(r) \diff r = \frac{1}{2\ep} \int_{-\ep}^{\ep} \Phi(0) - \left( \int_{-\infty}^{\infty} e^{-i\xi r} \diff \eta(\xi) \right) \diff t \\
& = \Phi(0) - \int_{-\infty}^{\infty} \left( \int_{-\ep}^{\ep} \frac{e^{-i\xi r}}{2 \ep}\diff r \right) \diff \eta(\xi) = \int_{-\infty}^{\infty} 1 - \left( \int_0^{\ep} \frac{\cos (\xi r)}{\ep} \diff r \right) \diff \eta(\xi) \\
&= \int_{-\infty}^{\infty} 1 - \frac{\sin (\ep\xi)}{\ep \xi} \diff \eta(\xi) = \int_{\left[-\frac{2}{\ep},\frac{2}{\ep} \right]} 1 - \frac{\sin (\ep\xi)}{\ep \xi} \diff \eta(\xi)  + \int_{\R \setminus \left[-\frac{2}{\ep},\frac{2}{\ep} \right]} 1 - \frac{\sin (\ep\xi)}{\ep \xi} \diff \eta(\xi) .
\end{split}
\end{equation*}
Since $ \sin(x)/x \leq \max\{1,1/|x|\}$ for all $x \in \R$, we get
$$
\frac{1}{2\ep} \int_{-\ep}^{\ep} \Phi(0) - \Phi(r) \diff r \geq  \int_{\R \setminus\left[-\frac{2}{\ep},\frac{2}{\ep} \right]} 1 - \frac{\sin (\ep\xi)}{\ep \xi} \diff \eta(\xi) \geq \int_{\R \setminus \left[-\frac{2}{\ep},\frac{2}{\ep} \right]} 1 - \frac{1}{|\ep \xi|} \diff \eta(\xi). 
$$
Clearly, $ \R \setminus \left[-\frac{2}{\ep},\frac{2}{\ep} \right] = \left\{ \xi \in \R : 1 - \frac{1}{|\ep \xi|} > \frac{1}{2}\right\}$, thus, by Chebyshev inequality, 
$$
\eta  \left(\R \setminus \left[-\frac{2}{\ep},\frac{2}{\ep} \right] \right) \leq 2 \int_{\R \setminus \left[-\frac{2}{\ep},\frac{2}{\ep} \right] } 1 - \frac{1}{|\ep \xi|} \diff \eta(\xi) \leq \frac{1}{\ep} \int_{-\ep}^{\ep} \Phi(0) - \Phi(r) \diff r,
$$
which proves the initial claim.


\section{Proofs of Lemma \ref{lemma:phi_h} and Proposition \ref{lemma:one_sided}}\label{sec:appendix2}

\subsection{Proof of Lemma \ref{lemma:phi_h}}

In this proof, we will enounter another loss in regularity, namely we will need to replace $\theta$ with a different one to make the observables Lipschitz: again, this is not an issue and we will freely do so.

Let $\psi \in \mathscr{L}$ and define $\psi_h(x,r) = \psi(x, r - h(x))$, where $h \in \mathscr{F}_{\theta}$.
By definition, it is clear that $\psi_h$ is well-defined from $\Sigma$ to $\mathscr{S}$, and, by the Fubini-Tonelli Theorem, $\nu(\psi) = \nu(\psi_h)$. The only thing to check is that it is a Lipschitz map.

Let $a,\ell$ be non negative integers, and fix $x,y \in \Sigma$. We need to estimate
$$
\| \psi_h(x) - \psi_h(y)\|_{a,\ell} := \sup_{r \in \R} |r|^a |\partial^\ell \psi_h(x,r) - \partial^\ell \psi_h(y,r)|.
$$
By definition, for any $r \in \R$, we have
\begin{equation*}
\begin{split}
|r|^a |\partial^\ell \psi_h(x,r) - \partial^\ell \psi_h(y,r)| \leq & |r|^a |\partial^\ell \psi(x,r - h(x)) -\partial^\ell \psi(x,r - h(y))| \\
&+ |r|^a |\partial^\ell \psi(x,r - h(y)) - \partial^\ell \psi(y,r - h(y))|\\
\end{split}
\end{equation*}
Let us consider the two summands separately. For the fisrt one, we have
$$
|r|^a |\partial^\ell \psi(x,r - h(x)) -\partial^\ell \psi(x,r - h(y))| \leq |r|^a |\partial^{\ell+1} \psi(x,r - h(x) + u) | \cdot |h(x) - h(y)|,
$$
for some $|u| \leq h(x)-h(y) \leq 2\|h\|_\infty$. 
If $|r|\leq 4\|h\|_\infty$, we have
$$
|r|^a |\partial^\ell \psi(x,r - h(x)) -\partial^\ell \psi(x,r - h(y))| \leq (4\|h\|_\infty)^a \|\psi(x)\|_{0,\ell+1} |h|_{\theta} d_{\theta}(x,y),
$$
otherwise, if $|r| > 4\|h\|_\infty$, then
$$
|r|^a |\partial^\ell \psi(x,r - h(x)) -\partial^\ell \psi(x,r - h(y))| \leq \|\psi(x)\|_{a,\ell+1} |h|_{\sqrt{\theta}} d_{\sqrt{\theta}}(x,y).
$$
In both cases, the first term satisfies a Lipschitz bound, independent of $r$.
For the second term, if  $|r|\leq 4\|h\|_\infty$, then 
$$
|r|^a |\partial^\ell \psi(x,r - h(y)) - \partial^\ell \psi(y,r - h(y))| \leq (4\|h\|_\infty)^a \|\psi(x) - \psi(y)\|_{0,\ell},
$$
and the Lipschitz bound follows from the assumption on $\psi$; otherwise, if $|r| > 4\|h\|_\infty$, then
$$
|r|^a |\partial^\ell \psi(x,r - h(y)) - \partial^\ell \psi(y,r - h(y))| \leq  \|\psi(x) - \psi(y)\|_{a,\ell},
$$
and again the conclusion follows from the assumption on $\psi$. 
This concludes the proof of the claims on $\psi$.

Let now $\Phi \in \mathscr{G}$. 
Then, for any $x \in \Sigma$, the function $\Phi_h(x)(r) = \Phi(x, r-h(x))$ is the Fourier-Stieltjes transform of the measure $\diff (\eta_h)_x = e^{i \xi h(x)}\diff \eta_x$. In particular, the variation are the same $|(\eta_h)_x| = |\eta_x|$ and, from Lemma \ref{lemma:nu_u}, it also follows that $\nuav(\Phi_h) = \nuav(\Phi)$.
Again, the only claim left to be shown is the Lipschitz assumption. We will exploit here the tail condition \eqref{eq:TC}.

Fix $x,y \in \Sigma$. Then,
$$
\|\Phi_h(x) - \Phi_h(y)\| = \| (\eta_h)_x - (\eta_h)_y \|_{\TV} \leq\| (e^{i \xi h(x)} - e^{i \xi h(y)} )\eta_x\|_{\TV} + \|\eta_x - \eta_y\|_{\TV}.
$$
The second summand in the right hand-side above satisfies a Lipschitz bound by assumption. We need to verify for the first term.
If $h(x) = h(y)$, the term is 0 and the proof is complete.
Assume $h(x) \neq h(y)$ and let $I = [-|h(x) - h(y)|^{-1/2}, |h(x) - h(y)|^{-1/2}]$. Then we have
\begin{equation*}
\begin{split}
& \| (e^{i \xi h(x)} - e^{i \xi h(y)} )\eta_x\|_{\TV} = \int_\R |1 - e^{i \xi (h(x)- h(y))}| \diff |\eta_x|(\xi) \\
& \qquad = \int_I |1 - e^{i \xi (h(x)- h(y))}| \diff |\eta_x|(\xi) + \int_{\R \setminus I} |1 - e^{i \xi (h(x)- h(y))}| \diff|\eta_x|(\xi).
\end{split}
\end{equation*}
the first term can be bound by 
$$
\int_I |1 - e^{i \xi (h(x)- h(y))}| \diff |\eta_x|(\xi) \leq |h(x) - h(y)|^{1/2} \|\eta_x\|_{\TV},
$$
and the second by 
$$
\int_{\R \setminus I} |1 - e^{i \xi (h(x)- h(y))}| \diff|\eta_x|(\xi) \leq 2 |\eta_x|(\R \setminus I) \leq  |h(x) - h(y)|^{a/2},
$$
hence the proof follows from the fact that $h$ is Lipschitz (again, up to possibly replacing $\theta$ with $\theta^{a/2}$ if $a<2$).

\subsection{Proof of Proposition \ref{lemma:one_sided}}

This section is devoted to the proof of Proposition \ref{lemma:one_sided}.
The strategy follows the same lines as in \cite{Dol} and \cite[Proposition 1.2]{PaPo}, although there are additional difficulties in showing that the functions $\Phi$ and $\psi$ belong to $\mathscr{G}^+$ and $\mathscr{L}^+$ respectively.

Let $\alpha \colon \Sigma \times \R \to \R$ be any Lipschitz function, with Lipschitz constant $L(\alpha)$. Fix $m \in \N$.
The first step is to prove that there exists a function $\beta\colon \Sigma \times \R \to \R$ such that the function
$$
\alpha^+ = (\alpha\circ F^m) + \beta \circ F - \beta
$$
depends only on the future coordinates. In other words, we want to show that $\alpha \circ F^m$ is cohomologous to a function defined on $X \times \R$.
We recall the construction of $\beta$ for the reader's convenience.
For any cylinder $\mathscr{C}_{n,j} := \mathscr{C}_{-n,o}(x^j)$, choose an element $\omega^{n,j} \in \mathscr{C}_{n,j}$.
Define the element $\omega^n(x) \in \Sigma$ in the following way:
$$
(\omega^n(x))_i = 
\begin{cases}
x_i & \text{ if } i \geq 0,\\
\omega^{n,j}_i & \text{ if } i <0.
\end{cases}
$$
Notice that by definition $d_{\theta}(\omega^n(x), x)\leq \theta^n$ for all $x \in \Sigma$.
Define 
$$
\alpha^{(n)}(x, r) := \alpha ( \omega^n(x), r + \delta_n(x) ), \text{\ \ \ where\ \ \ }\delta_n(x) := f_n(\omega^n(x))-f_n(x). 
$$
Since $\norm{\delta_n}_{\infty} \leq |f|_{\theta} (1-\theta)^{-1} \theta^n$, we have
\begin{equation}\label{eq:appendix1}
\norm{\alpha- \alpha^{(n)}}_{\infty} \leq L(\alpha) \left( d_{\theta}(\omega^n(x),x) + \norm{\delta_n}_{\infty}\right) \leq C_\theta L(\alpha) \theta^n,
\end{equation}
for a constant $C_\theta = 1 + |f|_{\theta} (1-\theta)^{-1}$.
Define 
$$
\beta(x,r) := \sum_{n =m}^{\infty} (\alpha \circ F^n - \alpha^{(n)} \circ F^n)(x,r),
$$
which, by \eqref{eq:appendix1}, is well-defined as a function from $\Sigma \times \R \to \R$. 
Then, we can define 
$$
\alpha^+ := \alpha \circ F^m + \beta \circ F - \beta = \sum_{n =m}^{\infty} \alpha^{(n)}\circ F^n - \alpha^{(n)}\circ F^{n+1},
$$
and from its definition, it is clear that $\alpha^+$ depends only on the future coordinates; namely, it follows that $\alpha \circ F^m$ is cohomologous to a function defined on $X \times \R$.

We will show in Lemma \ref{lemma:appendix1} and Lemma \ref{lemma:appendix2} below that if $\alpha$ belongs to $\mathscr{L}$ or $\mathscr{G}$, then $\alpha^+$ belongs to $\mathscr{L}^+$ or $\mathscr{G}^+$ respectively. Assuming these claims, let us first finish the proof, namely we show (i) and (iii).
There exist $\Phi^+ \in \mathscr{G}^+$ and $\psi^+ \in \mathscr{L}^+$ which are cohomologous to $\Phi \circ F^m$ and $\psi \circ F^m$ respectively. In particular, we have
$$
\nu(\psi^+) = \nu(\psi \circ F^m) = \nu(\psi),\text{\ \ \ and\ \ \ }\nuav(\Phi^+) = \nuav(\Phi \circ F^m) = \nuav(\Phi).
$$
Moreover, from the definitions of $\psi^+$ and $\Phi^+$ and from \eqref{eq:appendix1}, it follows that
$$
\| \Phi^+ - \Phi \circ F^m\| = O(\theta^m),
$$
and similarly for $\psi$. Therefore, by invariance of the measure $\nu_u$ under $F^m$, we obtain
$$
|\cov_{\Phi, \psi}(n)| = |\cov_{\Phi \circ F^m, \psi \circ F^m}(n) |= |\cov_{\Phi, \psi}(n)| +O(\theta^m),
$$
which concludes the proof.

\begin{lemma}\label{lemma:appendix1}
If $\alpha\in \mathscr{L}$, then $\alpha^+ \in \mathscr{L}^+$. For any $\ell \in \N$, there exist constants $M(A,\ell)$ and $L(A,\ell)$ depending on $\alpha$ and $\ell$ only such that $\Max_\ell(\alpha^+) \leq M(\alpha,\ell)$ and $\Lip_\ell(\alpha^+) \leq \theta^{-m}L(\alpha,\ell)$.
\end{lemma}
\begin{proof}
In order to prove the result, it is sufficient to show that for every $x \in \Sigma$, the function $\beta(x, \cdot)$ is Schwartz and the functions $x \mapsto \partial^\ell \beta(x,\cdot)$ are lipschitz with respect to the $L^1$-norm.

Since, for every $x \in \Sigma$ and $\ell \in \N$, the series 
$$
\sum_{n =m}^{\infty} \partial^\ell \alpha (\sigma^n x, r + f_n(x)) - \partial^\ell \alpha(\omega^n(\sigma^n x), r + f_n(x) + \delta_n(\sigma^n x)) 
$$
converges uniformly, the derivative $\partial^\ell \beta(x,\cdot)$ exists and equal the series above. We now show that $\beta(x,\cdot)$ is a Schwartz function for every $x \in \Sigma$.

Let $a, \ell \in \N$. We need to show that $r^a\partial^\ell \beta(x,r)$ is uniformly bounded in $r$.
To save notation, let us write $x_n = \sigma^nx$, $y_n = \omega^n(\sigma^n x)$ and $\delta_n = \delta_n(\sigma^n x)$. We have
\begin{equation*}
\begin{split}
|r^a\partial^\ell \beta(x,r)| \leq & \left\lvert  r^a  \sum_{n =m}^{\infty} \partial^\ell \alpha (x_n, r + f_n(x)) - \partial^\ell \alpha(y_n, r+f_n(x)) \right\rvert \\
&+ \left\lvert  r^a  \sum_{n =m}^{\infty} \partial^\ell \alpha(y_n, r+f_n(x))  -\partial^\ell \alpha(y_n, r + f_n(x) + \delta_n) \right\rvert
\end{split}
\end{equation*}
Each term in the sum above satisfies a Lipschitz bound exactly as in the proof of Lemma \ref{lemma:phi_h}. Since the terms $d(x_n,y_n)$ and $\delta_n$ can be bounded by $O(\theta^n)$, the series above converge.
Therefore, $|r^a\partial^\ell \beta(x,r)|$ is uniformly bounded for all $a,\ell \in \N$, hence $\beta(x,\cdot)$ is a Schwartz function.

The lipschitz bounds on the functions $x \mapsto \partial^\ell \beta(x,\cdot)$ with respect to the $L^1$-norm can be proved in a similar way and is left as an exercise to the reader: we remark that if $\alpha$ is a lipschitz function with constant $L(\alpha)$, then $\alpha \circ F^m$ is lipschitz with constant $L(\alpha \circ F^m) \leq \theta^{-m}L(\alpha)$.
\end{proof}

We now show the analogous result if we assume that $\alpha$ is a global observable.

\begin{lemma}\label{lemma:appendix2}
If $\alpha \in \mathscr{G}$, then $\alpha^+ \in \mathscr{G}^+$, and $\|\alpha^+\|_{\mathscr{G}^+} \leq \|\alpha \|_\mathscr{G}$.
\end{lemma}
\begin{proof}
It suffices to show that for any $x \in \Sigma$, $\beta(x, \cdot)$ is the Fourier-Stieltjes transform of a complex measure $\eta_x$ and moreover the total variation $\| \eta_x\|_{\TV}$ is uniformly bounded.

By definition,
$$
\beta(x,r) = \lim_{N \to \infty} \left( \sum_{n=m}^N \alpha (\sigma^n x, r + f_n(x)) - \alpha(\omega^n(\sigma^n x), r + f_n(x) + \delta_n(\sigma^n x))  \right).
$$
Fix $x \in \Sigma$ and again let us denote $x_n = \sigma^nx$, $y_n = \omega^n(\sigma^n x)$ and $\delta_n = \delta_n(\sigma^n x)$. 
For $N\geq m$, let $\zeta_N = \zeta_N^{(1)}+\zeta_N^{(2)} $ be the complex measure defined by 
\begin{equation*}
\begin{split}
&\diff \zeta_N(\xi) = \diff \zeta_N^{(1)}(\xi) + \diff \zeta_N^{(2)}(\xi) = \sum_{n=m}^N e^{i\xi f_n(x)} \diff \eta_{x_n}(\xi) - e^{i\xi (f_n(x) + \delta_n)} \diff \eta_{y_n}(\xi),\\
\text{where\ \ \ } & \diff \zeta_N^{(1)}(\xi) = \sum_{n=m}^N e^{i\xi f_n(x)} \diff \eta_{x_n}(\xi) - e^{i\xi f_n(x)} \diff\eta_{y_n}(\xi) \\
\text{and\ \ \ } &\diff \zeta_N^{(2)}(\xi)= \sum_{n=m}^N (e^{i\xi f_n(x)} - e^{i\xi (f_n(x) + \delta_n)}) \diff\eta_{y_n}(\xi).
\end{split}
\end{equation*}
From the expression for $\beta$ above, by definition, the Fourier-Stieltjes transforms $\widehat{\zeta_N}$ of $\zeta_N$ converge uniformly to $\beta(x,\cdot)$.
In order to conclude that $\beta(x, \cdot)$ is the Fourier-Stieltjes transform of a complex measure, it is enough to show that the family of measures $\zeta_N$ is contained in a weakly compact set. 
We will proceed as in the proof of Lemma \ref{lemma:phi_h} and use the tightness condition \eqref{eq:TC}.

We will show this separately for $\zeta_N^{(1)}$ and $\zeta_N^{(2)}$. 
For $\zeta_N^{(1)}$, by the lipschitz assumption on $\alpha$, we have
$$
\|\zeta_N^{(1)} \|_{\text{TV}} = \left\| \sum_{n=m}^N e^{i\xi f_n(x)} \eta_{x_n} - e^{i\xi f_n(x)}\eta_{y_n}\right\|_{\text{TV}} \leq \sum_{n=m}^N \|  \eta_{x_n} -\eta_{y_n}\|_{\text{TV}} \leq \|\alpha \|_\mathscr{G} \sum_{n=0}^N\theta^n.
$$
The total variation norm is stronger that the weak-convergence topology, hence the sequence of measures $\zeta_N^{(1)}$ converges weakly (since the tails are exponentially small). 
For the second term, we apply Prokhorov theorem: it suffices to show that the sequence $\zeta_N^{(2)}$ is uniformly bounded in total variation norm and is \emph{tight}.
Notice that the variation of $\zeta_N^{(2)}$ can be bounded by
$$
|\zeta_N^{(2)}| \leq \sum_{n=m}^N |1-e^{i\xi\delta_n}| |\eta_{y_n}|.
$$
Using the tightness condition \eqref{eq:TC},
\begin{equation*}
\begin{split}
& \|\zeta_N^{(2)}\|_{\text{TV}} = |\zeta_N^{(2)}| (\R) \\
& \leq \sum_{n=m}^N |1-e^{i\xi\delta_n}| |\eta_{y_n}| \big([-\theta^{-n/2}, \theta^{-n/2}] \big)+  \sum_{n=m}^N |1-e^{i\xi\delta_n}| |\eta_{y_n}| \big(\R \setminus [-\theta^{-n/2}, \theta^{-n/2}] \big)\\
& \leq \sum_{n=m}^N |\eta_{y_n}|(\R) \left(\max_{\xi \in [-\theta^{-n/2}, \theta^{-n/2}]} |1-e^{i\xi\delta_n}| \right) + 2\sum_{n=m}^N |\eta_{y_n}| \big(\R \setminus [-\theta^{-n/2}, \theta^{-n/2}] \big)\\
& \leq \|A\|_\mathscr{G} \sum_{n=m}^N  \|\delta_n\| \theta^{-n/2} + 2A \|A\|_\mathscr{G} \sum_{n=m}^N \theta^{an/2} \leq (1+2A)\|A\|_\mathscr{G}\sum_{n=m}^N (\theta^{n/2} + \theta^{an/2}).
\end{split}
\end{equation*}
This shows that $\zeta_N^{(2)}$ is uniformly bounded in total variation. Similarly we prove tightness: let us fix $\varepsilon >0$, and let $K = [-\varepsilon^{-1/a},  \varepsilon^{-1/a}] $. Then,
\begin{equation*}
\begin{split}
|\zeta_N^{(2)}| (\R \setminus K) \leq &\sum_{n=m}^N |1-e^{i\xi\delta_n}| |\eta_{y_n} | \big([-\varepsilon^{-1/2}\theta^{-n/2}, \varepsilon^{-1/2}\theta^{-n/2}] \cap (\R \setminus K) \big) \\
&+  \sum_{n=m}^N |1-e^{i\xi\delta_n}| |\eta_{y_n}| \big(\R \setminus [-\varepsilon^{-1/2}\theta^{-n/2}, \varepsilon^{-1/2}\theta^{-n/2}] \big)\\
\leq &\sum_{n=m}^N |\eta_{y_n}|(\R \setminus K) \|\delta_n\| \theta^{-n/2} \varepsilon^{-1/2}+ 2A \varepsilon^{a/2}\sum_{n=m}^N  \theta^{an/2} \\
 \leq & A\varepsilon^{1/2}\sum_{n=m}^N \theta^{n/2}+ 2A \varepsilon^{a/2}\sum_{n=m}^N  \theta^{an/2}.
\end{split}
\end{equation*}
This proves tightness and hence concludes the proof.
\end{proof}

\subsection*{Acknowledgements}
P.G.~acknowledges the support of the Centro di Ricerca Matematica Ennio de Giorgi and of UniCredit Bank R\&D group through the “Dynamics and Information Theory Institute” at the Scuola Normale Superiore.
This research was partially funded by the Australian Research Council.

We would like to thank Oliver Butterley, Dima Dolgopyat, Marco Lenci, Federico Rodriguez Hertz, Omri Sarig, and Khadim War for several useful discussions.


\begin{thebibliography}{20}

\bibitem{Aar} J.~Aaronson. {\it An introduction to infinite ergodic theory}. Mathematical Surverys and Monographs \textbf{50}, American Mathematical Society, 1997.

\bibitem{AaNa} J.~Aaronson, H.~Nakada. {\it On multiple recurrence and other properties of  ‘nice’ infinite measure-preserving transformations}. Ergodic Theory and Dynamical Systems, \textbf{37}:1345--1368, 2017.

\bibitem{ArMe} V.~Araujo, I.~Melbourne. {\it Exponential decay of correlations for nonuniformly hyperbolic flows with a $\mathscr{C}^{1+\alpha}$ stable foliation, including the classical Lorenz attractor}. Ann. Henri Poincaré  \textbf{17}(11):2975--3004, 2016.

\bibitem{ABV} V.~Araujo, O.~Butterley, P.~Varandas. {\it Open sets of axiom A flows with exponentially mixing attractors}. Proc. Amer. Math. Soc. \textbf{144}(7):2971--2984, 2016.

\bibitem{BeSi} A.~Berger, S.~Siegmund. {\it On the gap between random dynamical systems and continuous skew products}. Journal of Dynamics and Differential Equations, \textbf{15}(2-3):237--279, 2003.

\bibitem{Bob} S.G.~Bobkov. {\it Proximity of probability distributions in terms of their Fourier-Stieltjes transforms}. Russian Mathematical Surveys \textbf{71}(6):1021--1079, 2016.

\bibitem{Bog} N.N.~Bogoliubov. {\it Problems of a dynamical theory in statistical physics}. In: Studies in Statistical Mechanics \textbf{1}:1--118, 1962.


\bibitem{Bowen70} R. Bowen {\it Markov partitions for Axiom A diffeomorphisms}, American Journal of Mathematics, \textbf{92}:725--747, 1973.

\bibitem{Bowen73} R. Bowen {\it Symbolic dynamics for hyperbolic flows}, American Journal of Mathematics, \textbf{95}:429--460, 1973.

\bibitem{Bow78} R.~Bowen. {\it Markov partitions are not smooth}. Proc. Amer. Math. Soc.  \textbf{71}(1):130--132, 1978.

\bibitem{BowenBook} R.~Bowen. {\it Equilibrium States and the Ergodic Theory of Anosov Diffeomorphisms}. Lecture notes in Mathematics \textbf{420}, 2008.

\bibitem{BuEs} O.~Butterley, P.~Eslami. {\it Exponential mixing for skew-products with discontinuities}. Transaction of the American Mathematical Society, \textbf{369}(2):783--803, 2017.

\bibitem{BLG} C.~Bonanno, P.~Giulietti, M.~Lenci. {\it Infinite mixing for one-dimensional maps with an indifferent fixed point}. Nonlinearity, \textbf{31}(11):5180--5213, 2018.

\bibitem{BPSW} {K. {Burns}, C. {Pugh}, M. {Shub} and A. {Wilkinson}},
  {\it {Recent results about stable ergodicity.}},  {{Smooth ergodic theory and its applications. Proceedings of the AMS summer research institute, Seattle, WA, USA, July 26--August 13, 1999}}.
 
\bibitem{BW}  {K. {Burns}, A. {Wilkinson}},
  {\it {On the ergodicity of partially hyperbolic systems.}}, Annals of Mathematics \textbf{171}:451--489, 2010.

\bibitem{BuWa} O.~Butterley, K.~War. {\it Open sets of exponentially mixing Anosov flows}. J. Eur. Math. Soc. (JEMS) \textbf{22}(7): 2253--2285, 2020.

\bibitem{CLP} P.~Cirilo, Y.~Lima, E.~Pujals. {\it Ergodic properties of skew-products in infinite measure}. Israel Journal of Mathematics, \textbf{241}(1):43--66, 2016.

\bibitem{CFS} I.P.~Cornfeld, S.V.~Fomin, Ya.G.~Sinai. {\it Ergodic theory}. Grundlehren der Mathematicschen Wissenschaften \textbf{245}, 1982.

\bibitem{Dol} D.~Dolgopyat. {\it On mixing properties of compact group extensions of hyperbolic systems}. Israel Journal of Mathematics, \textbf{130}:157--205, 2002.

\bibitem{DLN} D.~Dolgopyat, M.~Lenci, P.~Nandori. {\it Global observables for random walks: law of large numbers}. Preprint arXiv:1902.11071, 2018.

\bibitem{DoNa} D.~Dolgopyat, P.~Nandori. {\it Infinite measure mixing for some mechanical systems}. Preprint arXiv:1812.01174, 2018.

\bibitem{DoNa2} D.~Dolgopyat, P.~Nandori. {\it Infinite measure renewal theorem and related results}. Bull. Lond. Math. Soc. \textbf{51}(1):145--167, 2019.

\bibitem{DoNaPe} D.~Dolgopyat, P.~Nandori, F.~Pene. {\it Asymptotic expansion of correlation functions for $\Z^d$ covers of hyperbolic flows}, preprint arXiv:1908.11504, 2019.

\bibitem{Fra} J.~Franks. {\it Anosov diffeomorphisms}. Global Analysis (Proc. Sympos. Pure Math.), Amer. Math. Soc., Providence, R.I, 1968.

\bibitem{GRS} {Galatolo, Stefano And Rousseau, Jérôme And Saussol, Benoit}   {\it Skew products, quantitative recurrence, shrinking targets and decay of correlations}, {Ergodic Theory and Dynamical Systems}, \textbf{35}:1814--1845, 2015.

\bibitem{GhLu} M.~Ghil, V.~Lucarini. {\it The Physics of Climate Variability and Climate Change}. Preprint arXiv:1910.00583, 2019.

\bibitem{Gou} S.~Gouezel. {\it Limit theorems in dynamical systems using the spectral method}. Proceedings of Symposia in Pure Mathematics, \textbf{89}:161--193, 2015.

\bibitem{Gou2} S.~Gouezel. {\it Correlation asymptotics from large deviations in dynamical systems with infinite measure}. Colloq. Math. 125:193--212, 2011.

\bibitem{Gui} Y.~Guivarc'h {\it Propri{\' e}t{\' e}s ergodiques, en mesure infinie, de certains syst{\` e}mes dynamiques fibr{\' e}s}. Ergodic Theory and Dynamical Systems \textbf{9}(3):433--453, 1989.

\bibitem{HPPS} M.~Hirsch, J.~Palis, C.~Pugh, M.~Shub. {\it Neighborhoods of hyperbolic sets}. Invent. Math. \textbf{9}:121--134, 1969/70.

\bibitem{HPS} M.~Hirsch, C.~Pugh, M.~Shub. {\it Invariant Manifolds}. Lecture Notes in Mathematics, Vol. 583. Springer-Verlag, Berlin-New York, 1977.

\bibitem{Hop} E.~Hopf. {\it Ergodentheorie}. Springer-Verlag, Berlin, 1937. 

\bibitem{Kat} T.~Kato. {\it Perturbation theory for linear operators}. Grundlehren der Mathematicschen Wissenschaften \textbf{132}, 1996.

\bibitem{Kri} K.~Krickeberg. {\it Strong mixing properties of Markov chains with infinite invariant measure}. In: 1967 Proc. Fifth Berkeley Sympos. Math. Statist. and Probability (Berkeley, Calif., 1965/66), Vol. II: Contributions to Probability Theory, Part 2: 431--446, Univ. California Press, Berkeley, CA, 1967. 


\bibitem{Len1} M.~Lenci. {\it On infinite volume mixing}. Communications in Mathematical Physics, \textbf{298}(2):485--514, 2010.

\bibitem{Len2} M.~Lenci. {\it Uniformly expanding Markov maps of the real line: exactness and infinite mixing}. Discrete and Continuous Dynamical Systems, \textbf{37}(7):3867--3903, 2017.

\bibitem{LST} E.~Liflyand, S.~Samko, R.~Trigub. {\it The Wiener algebra of absolutely convergent Fourier integrals: an overview
}, Anal. Math. Phys. \textbf{2}:1--68, 2012.

\bibitem{Mcc} B.M.~McCoy. {\it  Advanced statistical mechanics}.   International Series of Monographs on Physics \textbf{146}, 2010.

\bibitem{Mel} I.~Melbourne. {\it Mixing for invertible dynamical systems with infinite measure}. Stoch. Dyn. 15(2), 2015. 

\bibitem{MeTe1} I.~Melbourne, D.~Terhesiu. {\it Operator renewal theory and mixing rates for dynamical systems with infinite measure}. Invent. Math. 189:61--110, 2012.

\bibitem{MeTe2} I.~Melbourne, D.~Terhesiu. {\it Renewal theorems and mixing for non Markov flows with infinite measure}. Ann. Inst. Henri Poincar{\'e} Probab. Stat. 56(1):449--476, 2020.

\bibitem{PaPo} W.~Parry, M.~Pollicott. {\it Zeta functions and the periodic orbit structure of hyperbolic dynamics}. Asterisque, \textbf{187-188}, 1990.

\bibitem{Pen} F.~P{\`e}ne. {\it Mixing and decorrelation in infinite measure: the case of the periodic Sinai billiard}. Ann. Inst. Henri Poincar{\'e} Probab. Stat. 55(1): 378--411, 2019.

\bibitem{Rat} M.~Ratner. {\it Markov partitions for Anosov flows on $n$-dimensional manifolds}. Israel J. Math. \textbf{15}:92--114, 1973.

\bibitem{Roy} A.~Roy. {\it Symbolic dynamics and Markov partitions}. Bull. Amer. Math. Soc. \textbf{35}(1):1--56, 1998.

\bibitem{RS75}  M. Reed  and  B. Simon, {\it Methods of modern mathematical physics. {II}. {F}ourier analysis, self-adjointness},
{Academic Press [Harcourt Brace Jovanovich, Publishers], New
              York-London}, {1975}

\bibitem{Rue} D.~Ruelle. {\it  Thermodynamic formalism}.  Cambridge Mathematical Library, the mathematical structures of equilibrium statistical mechanics, 2004.

\bibitem{HHU}  Hertz, Federico and Hertz, Jana and Ures, Raúl  {\it Accessibility and stable ergodicity for partially hyperbolic diffeomorphisms with 1D-center bundle}, Inventiones mathematicae \textbf{172}:353--381, 2008.

\bibitem{Sar} O.~Sarig. {\it Introduction to the transfer operator method}. Available on the author webpage.

\bibitem{TrBe} R.M.~Trigub, E.S.~Bellinsky. {\it Fourier Analysis and Approximation of Functions}, Springer, 2004.

\end{thebibliography}
\end{document}